\documentclass[11pt]{amsart}
\usepackage{mathtools}
\usepackage{xcolor}
\usepackage{amssymb}
\usepackage{enumitem}
\usepackage[margin=1in]{geometry}
\usepackage[utf8]{inputenc}
\usepackage[cyr]{aeguill}
\usepackage[pdfdisplaydoctitle=true,
            colorlinks=true,
            urlcolor=blue,
            citecolor=blue,
            linkcolor=blue,
            pdfstartview=FitH,
            pdfpagemode= UseNone,
            bookmarksnumbered=true]{hyperref}
\usepackage{todonotes}
\usepackage[all]{xy}
\usepackage{stmaryrd}

\newtheorem{theorem}{Theorem}[section]
\newtheorem{corollary}[theorem]{Corollary}
\newtheorem{lemma}[theorem]{Lemma}
\newtheorem{proposition}[theorem]{Proposition}
\theoremstyle{definition}
\newtheorem{definition}[theorem]{Definition}

\newtheorem{example}[theorem]{Example}
\newtheorem{remark}[theorem]{Remark}
\DeclareMathOperator{\dist}{dist}
\DeclareMathOperator{\grad}{\nabla}

\newcommand{\Abs}[1]{\left|#1\right|}

\newcommand{\R}{\mathbb R}
\newcommand{\E}{\mathbb E}
\newcommand{\dx}{\lambda}
\newcommand{\Prob}{\operatorname{Prob}(M)}
\newcommand{\Probk}{\operatorname{Prob}^k(M)}
\newcommand{\Dens}{\operatorname{Dens}_+(M)}
\newcommand{\Densk}{\operatorname{Dens}^k_+(M)}
\newcommand{\Diff}{\operatorname{Diff}(M)}
\newcommand{\SDiff}{\operatorname{Diff}_{\dx}(M)}
\newcommand{\ProbRn}{\operatorname{Prob}(\R^n)}
\newcommand{\DensRn}{\operatorname{Dens}_+(\R^n)}
\newcommand{\GFR}{G^{\mathrm{FR}}}

\newcommand{\ntr}{\nabla^{\text{tr}}}
\newcommand{\nablabar}{\overline{\nabla}}
\newcommand{\nablatilde}{\tilde{\nabla}}
\newcommand{\ud}{\textup{d}}

\title{The $L^p$-Fisher-Rao metric and Amari-\u{C}encov $\alpha$-connections}

\author[]
{Martin Bauer, Alice Le Brigant, Yuxiu Lu, Cy Maor}

\address{M. Bauer: Florida State University and University of Vienna; A. Le Brigant: SAMM, Universit\'e Paris 1;  Y. Lu: Florida State University; C. Maor: Einstein Institute of Mathematics, The Hebrew University of
Jerusalem}
\email{bauer@math.fsu.edu, alice.le-brigant@univ-paris1.fr.com, yl18f@fsu.edu,
cy.maor@mail.huji.ac.il}

\date{\today}
\subjclass[2010]{%
}

\begin{document}

\begin{abstract}
We introduce a family of Finsler metrics, called the $L^p$-Fisher-Rao metrics $F_p$, for $p\in (1,\infty)$, which generalizes the classical Fisher-Rao metric $F_2$, both on the space of densities $\Dens$ and probability densities $\Prob$. 
We then study their relations to the Amari-\u{C}encov $\alpha$-connections $\nabla^{(\alpha)}$ from information geometry: on $\Dens$, the geodesic equations of $F_p$ and $\nabla^{(\alpha)}$ coincide, for $p = 2/(1-\alpha)$. 
Both are pullbacks of canonical constructions on $L^p(M)$, in which geodesics are simply straight lines.
In particular, this gives a new variational interpretation of $\alpha$-geodesics as being energy minimizing curves.
On $\Prob$, the $F_p$ and $\nabla^{(\alpha)}$ geodesics can still be thought as pullbacks of natural operations on the unit sphere in $L^p(M)$, but in this case they no longer coincide unless $p=2$. 
Using this transformation, we solve the geodesic equation of the $\alpha$-connection by showing that the geodesic are pullbacks of projections of straight lines onto the unit sphere,
and they always cease to exists after finite time when they leave the positive part of the sphere. 
This unveils the geometric structure of solutions to the generalized Proudman-Johnson equations, and generalizes them to higher dimensions.
In addition, we calculate the associate tensors of $F_p$, and study their relation to $\nabla^{(\alpha)}$.
\end{abstract}

\maketitle

\setcounter{tocdepth}{1}
\tableofcontents

\section{Introduction}
Information geometry is concerned with the study of spaces of probability densities as differentiable manifolds. Its first developments were mostly about the finite-dimensional geometry of parametric statistical models, for which the space of distributions can be identified with the parameter space. In 1945, Rao \cite{rao1945} showed that the Fisher information could be used to define a Riemannian metric on this space, and in 1982, \u{C}encov \cite{cencov1982} proved that it was the only metric invariant with respect to sufficient statistics, for families with finite sample spaces. The Fisher-Rao metric was also shown to induce well-known geometries on certain important statistical models, such as hyperbolic geometry on normal distributions \cite{atkinson1981}. 

Encompassing the Fisher-Rao metric, a richer geometric structure was introduced by \u{C}encov \cite{cencov1982} and Amari \cite{amari2000methods} on spaces of parametric probability distributions. The Amari-\u{C}encov structure relies on a family of affine connections called the $\alpha$-connections, denoted by $\nabla^{(\alpha)}$, that are dual with respect to the Fisher-Rao metric, and such that the $0$-connection is the Levi-Civita connection. The $\alpha$-connections arise naturally as an interpolating family between the so-called exponential and mixture connections $\nabla^{(1)}$ and $\nabla^{(-1)}$, for which exponential and mixture families are (dually) flat manifolds. These geometric tools relate to natural information-theoretic quantities such as the Kullback-Leibler divergence, and have been used in statistical inference, e.g. to express conditions for existence of consistent and efficient estimators, or to obtain a purely geometric interpretation of the famous Expectation-Maximization (EM) algorithm in the presence of hidden variables \cite{amari2016information}.

In parallel, infinite-dimensional information geometry tools have also been developed in the non-parametric setting, although arguably to a lesser extent. 
The non-parametric Fisher-Rao metric was introduced by Friedrich in 1991 \cite{friedrich1991} on the space of all probability densities. 
He showed that it yields the historical Fisher information metric when restricted to finite-dimensional submanifolds representing parametric statistical models, and that the geometry is spherical with constant curvature $1/4$. More than two decades later, it was proved to be the only metric (up to a multiplicative factor) invariant with respect to the action of sufficient statistics, namely diffeomorphic change of the support, just like in the finite-dimensional case \cite{ay2015information, bauer2016uniqueness}. 
In the infinite-dimensional setting, it is possible to work with diffeomorphisms of the support instead of the densities themselves, since the space of smooth densities on a compact manifold $M$ with respect to a volume form $\dx$ can be obtained as the quotient $\Diff/\SDiff$ of diffeomorphisms modulo  diffeomorphisms preserving $\dx$. 
Using this representation Khesin, Lenells, Misiolek and Preston \cite{khesin2013geometry} have shown in 2013 that the Fisher-Rao metric can be obtained as the quotient of a right-invariant homogeneous Sobolev $\dot H^1$-metric on $\Diff$, see also~\cite{modin2015generalized} and the recent overview article~\cite{khesininformation}.

The Amari-\u{C}encov structure induced by the $\alpha$-connections also received interest in the non-parametric setting. Giblisco and Pistone \cite{gibilisco1998connections} defined the exponential and mixture connections in this case, and showed that for $\alpha\in (-1,1)$, the interpolating connections can be defined through a $p$-root mapping to an $L^p$ sphere, for $p=\frac2{1-\alpha}$. 
Divergences and dualistic structures are investigated in the monograph of Ay, Jost, Lê and Schwachhöfer \cite{ay2017information}, although the $\alpha$-connections themselves are not directly considered there in the infinite-dimensional setting. 
See also \cite{newton2012infinite} for a definition of the $\alpha$-divergences and $\alpha$-connections in a Hilbert manifold settings.
In \cite{lenells2014amari}, Lenells and Misio\l{}ek study the $\alpha$-connections on diffeomorphisms and relate their geodesic equations to a well-known equation, the generalized Proudman--Johnson equation. 
Very recently, three authors of the present paper showed that these Proudman--Johnson equations, on the real line, could alternatively be seen as the geodesic equations of right-invariant Finsler metrics on the diffeomorphism group \cite{bauer2022geometric}, which were first introduced in \cite{cotter2020r}. 
This led to making a first link between $\alpha$-connections and a family of Finsler metrics, which we investigate further here.

\subsection{Main contributions}
The aim of the present paper is three-fold. First, to introduce and study the $L^p$-Fisher-Rao metrics on (probability) densities
$$F_p(a):=F_p({\mu},a)=\left(\int\Abs{\frac{a}{\mu}}^p\mu\right)^{\frac{1}{p}},$$
for $p\in(1,\infty)$ and any density $\mu$ and tangent vector $a$. Note, that is a family of Finsler metrics that conincides with the Fisher-Rao metric when $p=2$. Second, to give a precise and rigorous review of the Amari-\u{C}encov $\alpha$-connections in the infinite-dimensional setting, a new variational formulation of their corresponding geodesics, and explicit solution formulas for them. Finally, to make links between the two, distinguishing between the space of densities, the space of probability densities, and parametric statistical models.

Next we will describe the main contributions in more details: we study the $L^p$-Fisher-Rao geometry of (probability) densities through a mapping to the set of positive functions,
$$\Phi_p(\mu)=\left(\frac{\mu}{\dx}\right)^{1/p},$$
where $\dx$ is some background probability measure.
Just like the Fisher-Rao metric is the pullback of the standard $L^2$-metric via the square-root transform \cite{khesin2013geometry,bruveris2019geometry,gibilisco2020p},
we show that the $L^p$-Fisher-Rao metric is the pullback of the $L^p$-norm via the mapping $\Phi_p$, that we call by analogy the \emph{$p$-root transform} (Theorems~\ref{thm:proot} and \ref{thm:proot_Prob}). 
The $L^p$-Fisher-Rao geometry on the space of densities is therefore that of a flat space, as described in Corollary~\ref{cor:Lpgeometry}, and on the space of probability densities that of the $L^p$-sphere (Theorem~\ref{thm:proot_Prob}).   
The $p$-root transform (for $p = \frac{2}{1-\alpha}$) also presents an alternative way to define the $\alpha$-connections as pullbacks of the trivial connection of the vector space of functions (Theorems~\ref{thm:proot} and \ref{thm:proot_Prob}), as first shown by Gibilisco and Pistone \cite{gibilisco1998connections} for probability distributions, albeit with a slightly different construction.
The geometric differences between these constructions for the $L^p$-Fisher-Rao metric and the $\alpha$-connections, which we systematically study in this paper, are summarized in Figure~\ref{fig:p_root}.

Towards this aim, we show that the geodesic equations of $F_p$ and $\nabla^{(\alpha)}$ coincide on $\Dens$ (for $\alpha=1-2/p$), but not on $\Prob$ (see Theorems~\ref{thm:geo_alpha_dens}, \ref{thm:geo_lp_dens}, \ref{eq:alpha_conn_geo_Prob} and \ref{thm:LpFR_geo_prob});
similarly, on $\Dens$ the Chern connection induced by $F_p$ coincides with the $\alpha$-connection, while this no longer holds on $\Prob$ (Theorem~\ref{thm:chern_dens} and Remark~\ref{rem:chern-prob}).
This provides the novel variational formulation of these $\alpha$-connection geodesics.

We further use the $p$-root transform to obtain explicit solution formulas for $\alpha$-geodesics on densities and on probability densities:
for densities, we show in Corollary~\ref{cor:Lpgeometry} that geodesics are pullbacks of straight lines in $L^p$ space, whereas for probability densities we show in Theorem~\ref{prop:alpha-geod-prob} that they are pullbacks of projections of straight lines in $L^p$ onto the $L^p$-sphere. 
In the latter case the projection involves a time rescaling that is obtained as a solution of an ordinary differential equation.
Similar solutions of the geodesic equation of the $\alpha$-connection were obtained for finite sample space \cite[pp.~50-51]{ay2017information}.
In the infinite-dimensional case with a one-dimensional base manifold $M$, it gives an explicit solution (modulo a solution to an ODE) of the  generalized Proudman-Johnson equation, for a certain range of parameters, and to the generalization to higher-dimensional base manifolds by Lenells and Misio\l{}ek \cite{lenells2014amari}.
There, they proved the complete integrability of these equations for the flat case $\alpha = \pm 1$ by providing an explicit solution formula.
Similarly, the integrability for the case $\alpha = 0$ was shown in \cite{khesin2013geometry}.
Our results can thus be interpreted as complete integrability of the $\alpha$-geodesic equation for the whole range $\alpha \in (-1,1)$.

The results in the one-dimensional situation are in correspondence with the analysis of~\cite{kogelbauer2020global,sarria2013blow}, where a similar $p$-root transform was  used to study the generalized Proudman-Johnson equation. In these articles it was used as an ad-hoc simplification of some auxiliary equations; here we expose the geometry behind it, which also simplifies some of the authors' calculations, and generalize it to higher dimensions.
These connections are summarized in Section~\ref{sec:PDEs}.


Throughout this paper we work in the smooth category, i.e., all densities are assumed to be smooth, and the underlying space $M$ is assumed to be a smooth manifold.
This is mainly in order to avoid some technicalities, and most results work in much lower regularity.
For example, for all results not involving the action of $\Diff$, the underlying space $M$ can be simply a measurable space, and in many cases densities only need to be integrable.


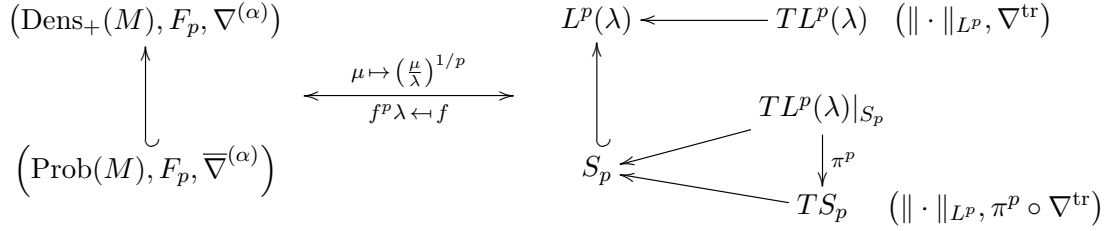
\begin{figure}
    \centering
    \[
\begin{xy}
(-60,0)*+{\left(\Prob,F_p,\overline{\nabla}^{(\alpha)}\right)}="Pr";
(-60,3)*+{}="Pr1";
(-60,20)*+{\left(\Dens,F_p,\nabla^{(\alpha)}\right)}="De";
(0,0)*+{S_p}="S";
(0,3)*+{}="S'";
(0,20)*+{L^p(\dx)}="Lp";
(30,20)*+{TL^p(\dx)}="TLp";
(50,20)*+{\left(\|\cdot\|_{L^p},\ntr\right)}="GLp";
(30,8)*+{TL^p(\dx)|_{S_p}}="TLpS";
(30,-5)*+{TS_p}="TS";
(53,-5)*+{\left(\|\cdot\|_{L^p},\pi^p\circ \ntr\right)}="GS";
{\ar@{->}^{\pi^p}"TLpS";"TS"};
{\ar@{->}^{}"TLp";"Lp"};
{\ar@{->}^{}"TLpS";"S"};
{\ar@{->}^{}"TS";"S"};
{\ar@{_{(}->}^{}"Pr1";"De"};
{\ar@{_{(}->}^{}"S'";"Lp"};
(-40,10)*+{}="x_1";
(-10,10)*+{}="x_2";
{\ar@{<->}^{\mu\,\mapsto \,\left(\frac{\mu}{\dx}\right)^{1/p}}_{f^p\dx\, \mapsfrom \,f}"x_1";"x_2"};
\end{xy}
\]
    \caption{Geometric structures on $\Dens$ and $\Prob$ via the $p$-root transform: The map $\mu \mapsto \left(\frac{\mu}{\dx}\right)^{1/p}$ maps $\Dens$ to (a subset of) $L^p(\dx)$, and $\Prob$ to its unit sphere $S_p$. 
    On $L^p(\dx)$ there is the natural Finsler metric $\|\cdot\|_{L^p}$ and the trivial connection $\nabla^{\text{tr}}$ of a vector space, the geodesics of both are straight lines. 
    Their pullback via the $p$-root map yield (up to a constant) the $L^p$-Fisher-Rao metric $F_p$ and the $\alpha$-connection $\nabla^{(\alpha)}$ on $\Dens$, whose geodesic equations coincide.  
    The metric $\|\cdot\|_{L^p}$ naturally restricts to $S_p$.
    The connection $\nabla^{\text{tr}}$ induces a connection on $S_p$ via the natural projection $\pi^p: TL^p(\dx)|_{S_p} \to TS_p$.
    The geodesics of these induced metric and connection differ.
    Their pullbacks via the $p$-root map yield (up to a constant) $F_p$ and the $\alpha$-connection $\overline{\nabla}^{(\alpha)}$ on $\Prob$.
    }
    \label{fig:p_root}
\end{figure}

\subsection{Outline}

The rest of the paper is organized as follows. We start  by describing some background on spaces of densities and the Fisher-Rao metric in Section~\ref{sec:background}. Then we investigate the geometries induced by the $\alpha$-connections and the $L^p$-Fisher-Rao metrics as well as their links, on the space of smooth densities in Section~\ref{sec:dens} and on the space of probability densities in Section~\ref{sec:prob}. 
In Section~\ref{sec:PDEs} we discuss the relations of the various geodesic equations obtained in Sections~\ref{sec:dens}--\ref{sec:prob} to some known PDEs, as well as the relation between the $L^p$-Fisher-Rao metric to Finsler metrics on diffeomorphism groups.
The different notions of geodesics are compared numerically on an example in Section \ref{sec:numerics}. 
Finally, we consider the finite-dimensional setting of parametric statistical models in Section~\ref{sec:finite-dim}, illustrated by the special case of normal distributions.
In Appendix~\ref{sec:finsler} we present a short overview of infinite-dimensional Finsler geometry.

\subsection*{Acknowledgements}
Parts of this work are contained in the PhD-thesis of the third author~\cite{Lu_thesis2023}.
MB was partially supported
by NSF grants DMS-1912037 and DMS-1953244 and by FWF grant FWF-P 35813-N. The first two authors acknowledge support of the Institut Henri Poincar{\'e} (IHP, UAR 839 CNRS-Sorbonne Universit{\'e}), and LabEx CARMIN (ANR-
10-LABX-59-01).
CM was partially supported by ISF grant 1269/19. 
\subsection*{Data availability statement}
Data sharing not applicable to this article as no datasets were generated or analysed during the current study.


\section{Spaces of densities and the Fisher-Rao metric}\label{sec:background}

In all of this article let $M$ be a closed manifold of dimension $\operatorname{dim}(M)<\infty$. 
We denote by $\Dens$ the space of smooth positive densities and by $\Prob$ the subspace of smooth probability densities, i.e., 
\begin{align*}
\Dens&:=\{\mu\in \Omega^n(M) : \mu>0\}\\
\Prob&:=\left\{\mu\in \Dens : \int\mu=1\right\}.
\end{align*}
Since $\Dens$ is an open subset of the Fr\'echet space $\Omega^n(M)$ it carries the structure of a Fr\'echet manifold with tangent space $T_{\mu}\operatorname{Dens}(M)=\Omega^n(M)$.
Similarly, as a linear subspace of a Fr\'echet manifold,  the space of probability densities is a Fr\'echet manifold, where the tangent space is given by 
\begin{align*}
T_{\mu}\Prob=\left\{a \in \Omega^n(M):\int a=0\right\}.
\end{align*}

On both the space of  densities and probability densities
we can consider the pushforward action of the diffeomorphism group $\Diff$.  On $\Dens$ it is given by
\begin{equation}\label{eq:pushforward}
\Diff\times \Dens\ni (\varphi,\mu)\mapsto \varphi_*\mu \in \Dens
\end{equation}
and, since the pushforward by a diffeomorphism is volume preserving, this action restricts to an action on the space of probability densities. 
By a result of Moser~\cite{moser1965volume} this action is transitive, which allows us to identify the space of probability densities with the quotient
\begin{equation}
\Prob\equiv \Diff/\SDiff, 
\end{equation}
where $\SDiff$ is the group of volume preserving diffeomorphisms of some fixed probability density $\dx$. 
Thus, constructions (metrics, connections, geodesics) on $\Prob$ can be pulled back to $\Diff$ via the map $\varphi\mapsto \varphi_*\dx$.

For $a\in \Omega^n(M)$ and $\mu\in \Dens$, we denote by $\frac{a}{\mu}$ the Radon-Nikodym derivative of $a$ with respect to $\mu$.
In particular, the map $\mu \mapsto \frac{\mu}{\dx}$ allows us to identify $\Dens$ with positive smooth functions on $M$, and $\Prob$ with the positive smooth functions that integrate to one.
For the proof of the local wellposedness results in Sections~\ref{sec:dens} and~\ref{sec:prob} we will also need the Sobolev completions of these spaces, which can be defined using their  Radon-Nikodym derivative w.r.t.\ to $\dx$, i.e., for $k>\operatorname{dim}(M)/2$ we consider
\begin{align*}
\Densk&:=\left\{\mu:\frac{\mu}{\dx}\in H^{k}(M), \text{ and } \mu>0\right\}\\
\Probk&:=\left\{\mu\in \Densk : \int\mu=1\right\}.
\end{align*}
Note, that the assumption  $k>\operatorname{dim}(M)/2$ is necessary to make sense of the positivity condition.

A central object in information geometry is the Fisher-Rao metric, which we introduce now:
\begin{definition}[Fisher-Rao metric]
Given $\mu\in \Dens$ and $a,b\in T_\mu\Dens$ the Fisher-Rao metric on $\Dens$ is given by
\begin{align}\label{fisher-rao}
\GFR_{\mu}(a,b)=\int \frac{a}{\mu} \frac{b}{\mu}\mu\;.
\end{align}
Via restriction $\GFR$ induces a Riemannian metric on $\Prob$, which we denote by the same letter.
\end{definition}


\section{The $L^p$-Fisher-Rao metric and $\alpha$-connections on the space of densities}\label{sec:dens}
In this section we will introduce the $L^p$-Fisher-Rao metric on the space of densities, which will allow us to obtain a new interpretation of the family of $\alpha$-connections.


\subsection{The Amari-\u{C}encov $\alpha$-connections on $\Dens$}
First we will introduce the family of $\alpha$-connections on the space $\Dens$.
In the finite-dimensional case, i.e., when $M$ is a finite set, the below definitions coincide with the classical ones, see e.g. \cite{amari2000methods,ay2015information}.
\begin{definition}[$\alpha$-divergence]
    For $\alpha\in (-1,1)$, define the $\alpha$-divergence $D^{(\alpha)}:\Dens\times \Dens\to \mathbb{R}$, by
    \[
    D^{(\alpha)}(\mu||\nu) = p\int_M \nu + p^*\int_M \mu -p^*p\int_M \left(\frac{\mu}{\dx}\right)^{1/p}\left(\frac{\nu}{\dx}\right)^{1/p^*}\dx,
    \]
    where $p=\frac{2}{1-\alpha}$ and $p^* = \frac{2}{1+\alpha}$ is its H\"older conjugate.
\end{definition}  
Using H\"older inequality, it follows that $D^{(\alpha)}$ is non-negative and vanishes if and only if $\mu=\nu$.
Furthermore, a straightforward calculation shows that the negative of its second derivative defines a positive bilinear form, which is exactly the Fisher-Rao metric, i.e., 
\[
 -\partial_\mu \partial_\nu D^{(\alpha)}(\mu||\nu)|_{\nu=\mu}[a,b] = \int_M \frac{a}{\mu}\frac{b}{\mu}\mu=\GFR_\mu(a,b),
\qquad a,b\in T\mu\Dens.
\]
Here $\partial_\mu$ and $\partial_\nu$ refer to derivatives with respect to the $\mu$ and $\nu$ variables, respectively.
Thus, for any $\alpha\in (-1,1)$,  $D^{(\alpha)}$ is a divergence in the sense of \cite[Section~4.4]{ay2015information}, and  induces a connection $\nabla^{(\alpha)}$ on $\Dens$ via the relation
\begin{equation}\label{eq:a_conn_def}
\GFR_\mu(\nabla^{(\alpha)}_a b, c) 
= - \partial_\mu(\partial_\mu\partial_\nu D^{(\alpha)}(\mu||\nu)[b,c])[a]|_{\nu=\mu} 
= \int_M \frac{Db.a}{\mu}\frac{c}{\mu}\mu-\frac{1}{p^*}\int_M \frac{a}{\mu}\frac{b}{\mu}\frac{c}{\mu}
\mu,
\end{equation}
where $a,b,c\in T_\mu\Dens$.
Since $\Dens$ is a Fr\'echet manifold, $\GFR$ is merely a weak Riemannian metric, and as such, \eqref{eq:a_conn_def} does not necessarily define $\nabla^{(\alpha)}$ uniquely.
However, in our case it does, yielding the following formulae:
\begin{lemma}[$\alpha$-connection]\label{lem:alphacon}
    For any $\alpha\in (-1,1)$ the $\alpha$-connections $\nabla^{(\alpha)}$ on $\Dens$ are given by
    \begin{equation}\label{alpha-dens}
    \nabla^{(\alpha)}_a b = Db.a - \frac{1}{p^*} \frac{a}{\mu}b,\quad a,b\in T_\mu\Dens,\;p^* = \frac{2}{1+\alpha}.
    \end{equation}
    Here $Db.a|_{\mu}:=D_\mu b(a_\mu)$ denotes the directional derivative of the vector field $b$ in the direction given by $a_\mu$. 
\end{lemma}
The easiest way to read this lemma (and similar formulae below) is to consider again the identification of densities and positive functions via $\mu\mapsto \mu/\dx$.
\begin{proof}
This follows directly from formula~\eqref{eq:a_conn_def}.
\end{proof}
In the following result we will study the local wellposedness of the corresponding geodesic equations. Therefore we will first consider these equations on a Banach space of Sobolev densities, where it will be easy to obtain the local wellposedness using the theorem of Picard-Lindel\"off. The result in the smooth category will then follow from an Ebin-Marsden type no-loss-no-gain result~\cite{ebin1970groups}:
\begin{theorem}\label{thm:geo_alpha_dens}
A path $\mu:[0,1]\to \Dens$ is a geodesic with respect to $\nabla^{(\alpha)}$ if
\begin{equation}\label{eq:geo_alphacon}
\mu_{tt} = \frac{1}{p^*}\frac{\mu_t}{\mu}\mu_t.
\end{equation}
For any $k>\operatorname{dim}(M)/2$ the geodesic equations are locally wellposed on the space of Sobolev densities $\Densk$, i.e., given initial conditions $\mu(0)\in \Densk$, $\mu_t(0)\in T_{\mu(0)}\Densk$ there exists an unique solution to equation~\eqref{eq:geo_alphacon} defined on a maximal interval of existence $[0,T)$. 
The maximal interval of existence is uniform in the Sobolev order $k$ and thus the local wellposedness continues to hold in the limit, i.e., on the space of smooth densities $\Dens$.
\end{theorem}

\begin{proof}
The formula for the geodesic equation follows directly from Lemma~\ref{lem:alphacon}. To show the local well-posedness we view the geodesic equation~\eqref{eq:geo_alphacon} as a flow equation on $T\Densk$. Therefore let $F(\mu_t)$ denote the right hand side of the geodesic equation, i.e., 
\begin{equation}
F(\mu,\mu_t)=\mu^{-1}\mu_t^2
\end{equation}
where we use the identification of $\Densk$ with the space of positive, Sobolev functions $H^k_+(M)$ and $T_{\mu}\Densk$ with all of $H^k(M)$. Using the Sobolev  module properties and the positivity of $\mu$ it follows that $F$ is a smooth map from $H^k_+(M)\times H^k(M)$ and thus the local well-posedness follows by the theorem of Picard-Lindel\"off. Next, we observe that $F$ is equivariant under the action of the diffeomorphism group $\Diff(M)$, i.e., $F(\varphi^* \mu,\varphi^* \mu_t)=\varphi^*F( \mu,\mu_t)$. Thus the result on the uniformness of the maximal interval of existence follows by an adaption of the Ebin-Marsden no-loss-no-gain theorem~\cite[Lemma 12.2]{ebin1970groups} to the present setting, i.e., the diffeomorphism group acting on densities. This can be achieved by following the proof in~\cite{bauer2022smooth}, where the no-loss-no-gain result has been extended to the action of  diffeomorphisms on the space of all Riemannian metrics. The
key ingredient for this result is the fact that, in a chart, Lie derivatives along coordinate vector fields coincide with ordinary derivatives. 
\end{proof}

\subsection{The $L^p$-Fisher-Rao metric}
Next we introduce the main object of the present article, the $L^p$-Fisher-Rao (Finsler) metric on the space $\Dens$:
\begin{definition}
Given $\mu\in \Dens$ and $a\in T_\mu\Dens$ we define the $L^p$-Fisher-Rao metric via:
\begin{align}\label{lp-fisher-rao}
F_p(a):=F_p({\mu},a)=\left(\int\Abs{\frac{a}{\mu}}^p\mu\right)^{\frac{1}{p}}.
\end{align}
\end{definition}
\begin{remark}
It is easy to see that the $L^p$-Fisher-Rao metric satisfies the axioms of a Finsler metric, as defined in Definition~\ref{def:Finsler}, for any $p\in(1,\infty)$. We will, however, see in~Lemma~\ref{lem:Hessian}, that it is not strongly convex if $p\neq 2$. 
\end{remark}

First we will show, that the family of $L^p$-Fisher-Rao metrics shares an important property with the Fisher-Rao metric: they are invariant under the action of the diffeomorphism group as defined in \eqref{eq:pushforward}.
\begin{lemma}\label{lem:invariance}
For any $p\in(1,\infty)$, the $L^p$-Fisher-Rao metric on the space of $\Dens$ is invariant under the action of the diffeomorphism group $\Diff$, i.e.,
\begin{equation}
F_p(\mu,a)=F_p(\varphi_*\mu,\varphi_*a),\qquad a\in T_\mu \Dens,\;\varphi\in\Diff.
\end{equation}
\end{lemma}
\begin{proof}
This result follows by direct computation using the transformation formula for integrals.
\end{proof}

Next we calculate the geodesic equations of this family of Finsler metrics on 
$\Dens$.
\begin{theorem}[Geodesic equation on $\Dens$]\label{thm:geo_lp_dens}
For any $p\in(1,\infty)$, the geodesic equation of the $L^p$-Fisher-Rao metric on the space of densities $\Dens$ is given by
\begin{align}\label{eq:geodesicDens} 
\frac{d}{dt}\left(\frac{\mu_t}{\mu}\right)+\frac{1}{p}\left(\frac{\mu_t}{\mu}\right)^2=0,
\end{align}
which coincides with the geodesic equation of the $\alpha$-connection for $\alpha=1-\frac{2}{p}$.
Thus the local well-posedness result of Theorem~\ref{thm:geo_alpha_dens} also hold for the geodesic equation of the $L^p$-Fisher-Rao metric.
\end{theorem}

\begin{proof}
The length functional of the $L^p$-Fisher-Rao metric on $\Dens$ is given by
\begin{align*}
L(\mu)=\int_0^1 \left(\int\Abs{\frac{\mu_t}{\mu}}^p\mu\right)^{\frac{1}{p}} dt,   
\end{align*}
where $\mu: [0,1]\to \Dens$ such that $\mu(0)=\mu_0$, $\mu(1)=\mu_1$ and where $\mu_t$ denotes its (time) derivative. 
A geodesic is a path that locally minimizes the length functional; since $L$ is invariant to reparametrization, we can restrict ourselves to paths of constant speed.
By the H\"older inequality, it follows that constant speed geodesics are equivalently the local minimizers of the $q$-energy 
\begin{align*}
E_q(\mu)=\frac1q\int_0^1 \left(\int\Abs{\frac{\mu_t}{\mu}}^p\mu\right)^{\frac{q}{p}} dt,   
\end{align*}
for any $q> 1$.
In our case the most convenient choice is to consider the $q$-Energy with $q=p$. 
The corresponding energy functional reads as
\begin{align*}
E_p(\mu)=\frac1p\int_0^1 \int\Abs{\frac{\mu_t}{\mu}}^p\mu \,dt.   
\end{align*}
Calculating the variation of the $p$-energy functional in direction $\delta \mu$ leads to
\begin{equation}\label{eq:deltaE}
\begin{aligned}
\delta E_p(\mu)(\delta \mu)&=\frac{1}{p}\int_0 ^1 \int  p\Abs{\frac{\mu_t}{\mu}}^{p-2}\frac{\mu_t}{\mu}\delta \mu_t-(p-1)\Abs{\frac{\mu_t}{\mu}}^p  \delta \mu \, \ud\dx\,dt\\ 
&=-\frac{1}{p}\int_0 ^1 \int  \left( p\frac{d}{dt}\left(\Abs{\frac{\mu_{t}}{\mu}}^{p-2}\frac{\mu_t}{\mu}\right)+(p-1)\Abs{\frac{\mu_t}{\mu}}^p \right) \, \delta \mu \, \ud\dx\,dt,
\end{aligned}
\end{equation}
where we used integration by parts in time $t$ and that the variational direction vanishes at the end points, i.e., $\delta \mu(0)=\delta\mu(1)=0$.
From here we can immediately read off the geodesic equation
\begin{align*} 
p\frac{d}{dt}\left(\Abs{\frac{\mu_{t}}{\mu}}^{p-2}\frac{\mu_t}{\mu}\right)+(p-1)\Abs{\frac{\mu_t}{\mu}}^p=0.
\end{align*}
which can be simplified to the desired formula. That this equation coincides with the geodesic equation of the $\alpha$-connection can be seen by comparing it to the equation of Theorem~\ref{thm:geo_alpha_dens}. 
\end{proof}

Next we will study the Finslerian geometry induced by the $L^p$-Fisher-Rao metric (see Appendix~\ref{sec:finsler} for a short overview of the main definitions). We will see in the next Lemma, that the $L^p$-Fisher-Rao metric is, in general, not strongly convex and thus some of the calculations in this and the next sections have to be understood formally.

\begin{lemma}[The Hessian matrix]\label{lem:Hessian}
Let $\mu\in\Dens$ and $\nu,a,b\in T_\mu\Dens$. The Hessian matrix $g^{\nu}$ of the squared $L^p$-Fisher-Rao metric at $\nu$ is given by
    \begin{equation}\label{riem-metric}
        g_\mu^\nu(a,b)=(p-1)I(\nu,\nu)^{\frac{2}{p}-1}I(a,b) - (p-2)I(\nu,\nu)^{\frac{2}{p}-2}I(\nu,a)I(\nu, b).
    \end{equation}
    where
    \begin{equation}
    \begin{aligned}
        I(a,b)&:=I_\mu^\nu(a,b):=\int \left|\frac{\nu}{\mu}\right|^{p-2}\frac{a}{\mu}\frac{b}{\mu}\mu\;.
    \end{aligned}   
    \end{equation}
If $\nu$ is nowhere zero  than  $g^\nu$ is positive definite and thus a Riemannian metric. If $\nu$ vanishes on an open set $\mathcal U\subset M$ then, for $p>2$, $g^\nu$ is degenerate as it vanishes for all $a,b\in T_\mu\Dens$ with support contained in $\mathcal U$, and for $p<2$ it is not well-defined.
\end{lemma}
\begin{proof}
We introduce the notations 
\begin{align*}
\tilde\omega=\tilde\omega(r,s):=\frac{\nu}{\mu}+r\frac{a}{\mu}+s\frac{b}{\mu},
\end{align*}
\begin{align*}
\omega:=|\tilde\omega(r,s)|=\left|\frac{\nu}{\mu}+r\frac{a}{\mu}+s\frac{b}{\mu}\right|.
\end{align*}
To compute the Hessian matrix of $F_p^2(\mu,\nu)$ we need to calculate the second derivative in $r$ and $s$ of $F_p^2(\mu,\omega)$.
We have
\begin{align*}
\partial_r F_p^2(\mu,\omega)= 2\left(\int\omega^p\mu\right)^{2/p-1} \int\omega^{p-1} \partial_r \omega \ \mu
=2            
\left(\int\omega^p\mu\right)^{2/p-1} \int\omega^{p-1} \operatorname{sgn} (\tilde\omega)\frac{a}{\mu} \ \mu.
\end{align*}
For the second derivative we get
\begin{multline*}
\partial_s \partial_r F_p^2(\mu,\omega)
=2(2-p)  \left(\int\omega^p\mu\right)^{2/p-2}  \int\omega^{p-1} \operatorname{sgn} (\tilde\omega)\frac{a}{\mu} \ \mu \int\omega^{p-1} \operatorname{sgn} (\tilde\omega)\frac{b}{\mu} \ \mu\\+
2(p-1) \left(\int\omega^p\mu\right)^{2/p-1} \int\omega^{p-2} \frac{a}{\mu} \frac{b}{\mu} \ \mu.
\end{multline*}  
Evaluating at $r=s=0$ yields the desired formula for $g^{\nu}_\mu$.

For $\nu\neq 0$ we can use the Cauchy-Schwarz inequality to prove the positive-definiteness of the Hessian:
\begin{align*}
\left(\int\Abs{\frac{\nu}{\mu}}^{p-1} \Abs{\frac{a}{\mu}} \ \mu\right)^2\leq \int\Abs{\frac{\nu}{\mu}}^p\mu \int\Abs{\frac{\nu}{\mu}}^{p-2} \left(\frac{a}{\mu}\right)^2 \ \mu.
\end{align*}
Then we get the inequality
\begin{align*}
g^{\nu}(a,a)\geq\left(\int\Abs{\frac{\nu}{\mu}}^p\mu\right)^{2/p-1} \int\Abs{\frac{\nu}{\mu}}^{p-2} \left(\frac{a}{\mu}\right)^2 \ \mu.
\end{align*}
Thus for $\nu$ being a nowhere vanishing vector field,
$g^{\nu}(a,a)=0$ implies that $\frac{a}{\mu}=0$.
\end{proof}

\begin{lemma}[The Cartan tensor]
    Let $\mu\in\Dens$ and $\nu,a,b,c\in T_\mu\Dens$. The Cartan tensor  of the $L^p$-Fisher-Rao metric is given by
    \begin{equation}\label{cartan-tensor}
        \begin{aligned}
        C^\nu(a,b,c) =& \frac{1}{2}(p-1)(p-2)I(\nu,\nu)^{\frac{2}{p}-3}\Big(2I(\nu,a)I(\nu,b)I(\nu,c) \\
        &- I(\nu,\nu)I(\nu,a)I(b,c) - I(\nu,\nu)I(\nu,b)I(c,a)\\
        &- I(\nu,\nu)I(\nu,c)I(a,b) + I(\nu,\nu)^2J(\nu,a,b,c)\Big),
        \end{aligned}
    \end{equation}
    where
    \begin{equation}\label{h-and-l}
    \begin{aligned}
        I(a,b)&:=I_\mu^\nu(a,b):=\int \Abs{\frac{\nu}{\mu}}^{p-2}\frac{a}{\mu}\frac{b}{\mu}\mu,\\
        J(a,b,c)&:=J_\mu^\nu(a,b,c):=\int\Abs{\frac{\nu}{\mu}}^{p-4}\frac{a}{\mu}\frac{b}{\mu}\frac{c}{\mu}\mu.
    \end{aligned}
    \end{equation}
\end{lemma}
\begin{proof}
This formula can be derived similarly as the formula for the Hessian
by computing
$$C^\nu(a,b,c)=\left.\partial_r\partial_s\partial_t\right|_{r=s=t=0}F_p(\mu,\tilde\omega)$$
where $\tilde \omega(r,s,t)=\nu+ra+sb+tc$.
\end{proof}

\subsection{The $\alpha$-connection as Chern connection of the $L^p$-Fisher-Rao metric}
Next we will show that the Chern connection associated to the $L^p$-Fisher-Rao metric on $\Dens$ is an $\alpha$-connection, when two entries are taken to be the same. 
\begin{theorem}[The Chern connection on $\Dens$]\label{thm:chern_dens}
    Let $\alpha=1-\frac{2}{p}$. For every nowhere vanishing vector field $\nu$ on $\Dens$ and any $a\in T_\mu\Dens$, we have
    \begin{equation}\label{chern-dens}
    \nabla_a^\nu\nu=\nabla^{(\alpha)}_a\nu,
    \end{equation}
    where $\nabla^\nu$ is the Chern connection induced by the $L^p$-Fisher-Rao metric and $\nabla^{(\alpha)}$ is the $\alpha$-connection on $\Dens$ defined by \eqref{alpha-dens}.
\end{theorem}
\begin{proof}
Formula \eqref{chern-dens} defines the Chern connection if and only if it verifies the generalized Koszul formula (see Lemma \ref{lem:koszul})
\begin{align*}
2g^\nu(\nabla^\nu_a\nu, b) &= ag^\nu(\nu, b)+\nu g^\nu(b, a) - bg^\nu(a,\nu) + g^\nu([a,\nu],b) - g^\nu([\nu,b],a) + g^\nu([b,a],\nu)\\
&-2C^\nu(\nabla^\nu_a\nu, \nu, b) - 2C^\nu(\nabla^\nu_\nu\nu, b, a) + 2C^\nu(\nabla^\nu_\nu b,a,\nu).
\end{align*}
Since the Cartan tensor verifies $C^\nu(\nu,\cdot,\cdot)=0$ this formula reduces to
\begin{equation}\label{koszul2}
\begin{aligned}
2g^\nu(\nabla^\nu_a\nu, b) &= ag^\nu(\nu, b)+\nu g^\nu(b, a) - bg^\nu(a,\nu) \\
&+ g^\nu([a,\nu],b) - g^\nu([\nu,b],a) + g^\nu([b,a],\nu) -2C^\nu(\nabla^\nu_\nu\nu, b, a).
\end{aligned}
\end{equation}
To compute the first terms of the right hand-side of this equality, we will need 
$$cI(a,b)=I(D a.c,b)+I(a,D b.c)-(p-1)K(a,b,c)+(p-2)J(\nu,a,b,D \nu.c),$$
where $I$ and $J$ are defined by \eqref{h-and-l}, and
$$K(a,b,c):=K^\nu_\mu(a,b,c)=\int \Abs{\frac{\nu}{\mu}}^{p-2}\frac{a}{\mu}\frac{b}{\mu}\frac{c}{\mu}\mu.$$
Using this we get
\begin{align*}
ag^\nu(\nu,b)&=-\frac{p-2}{p}I(\nu,\nu)^{2/p-2}I(\nu,b)\left(pI(\nu,D\nu.a)-(p-1)K(\nu,\nu,a)\right)\\
&+I(\nu,\nu)^{2/p-1}\left(I(\nu,D b.a)-(p-1)K(\nu,a,b)+(p-1)I(b,D\nu.a)\right),\\
bg^\nu(\nu,a)&=-\frac{p-2}{p}I(\nu,\nu)^{2/p-2}I(\nu,a)\left(pI(\nu,D\nu.b)-(p-1)K(\nu,\nu,b)\right)\\
&+I(\nu,\nu)^{2/p-1}\left(I(\nu,D a.b)-(p-1)K(\nu,a,b)+(p-1)I(a,D\nu.b)\right),
\end{align*}
and
\begin{align*}
&\nu g^\nu(a,b)=-\frac{(p-1)(p-2)}{p}I(\nu,\nu)^{2/p-2}I(a,b)\left(pI(\nu,D\nu.\nu)-(p-1)K(\nu,\nu,\nu)\right)\\
&+(p-1)I(\nu,\nu)^{2/p-1}\left(I(D a.\nu,b)+I(a,D b.\nu)-(p-1)K(\nu,a,b)+(p-2)J(\nu,a,b,D\nu.\nu)\right)\\
&+\frac{2(p-1)(p-2)}{p}I(\nu,\nu)^{2/p-3}I(\nu,a)I(\nu,b)\left(pI(\nu,D\nu.\nu)-(p-1)K(\nu,\nu,\nu)\right)\\
&-(p-2)I(\nu,\nu)^{2/p-2}I(\nu,b)\left(I(\nu,D a.\nu) - (p-1)K(\nu,\nu,a) + (p-1)I(a,D\nu.\nu)\right)\\
&-(p-2)I(\nu,\nu)^{2/p-2}I(\nu,a)\left(I(\nu,D b.\nu) - (p-1)K(\nu,\nu,b) + (p-1)I(b,D\nu.\nu)\right).
\end{align*}
The following terms of the right hand-side of the generalized Koszul formula \eqref{koszul2} are given by
\begin{align*}
g^\nu([a,\nu],b)&=(p-1)I(\nu,\nu)^{2/p-1}I(D\nu.a-D a.\nu,b)\\
&-(p-2)I(\nu,\nu)^{2/p-2}I(\nu,D\nu.a-D a.\nu)I(\nu,b),\\
g^\nu([\nu,b],a)&=(p-1)I(\nu,\nu)^{2/p-1}I(D b.\nu-D\nu.b,a)\\
&-(p-2)I(\nu,\nu)^{2/p-2}I(\nu,D b.\nu-D\nu.b)I(\nu,a),\\
g^\nu([b,a],\nu)&=I(\nu,\nu)^{2/p-1}I(\nu,D a.b-D b.a).
\end{align*}
Finally there remains to compute the two terms involving the Chern connection, i.e. the term on the left hand-side and the last term of the right hand-side. With the chosen value of $\alpha$, we have
$$\nabla^\nu_a\nu=D\nu.a-\frac{p-1}{p}\frac{a}{\mu}\frac{\nu}{\mu}\mu,$$
and so
\begin{align*}
I(\nabla^\nu_a\nu,b)&=I(D\nu.a,b)-\frac{p-1}{p}K(\nu,a,b)\\
J(\nu,a,b,\nabla^\nu_\nu\nu)&=J(\nu,a,b,D\nu.\nu)-\frac{p}{p-1}K(\nu,a,b).
\end{align*}
This yields, using \eqref{cartan-tensor},
\begin{align*}
2C^\nu(\nabla^\nu_\nu\nu,b,a)&=(p-1)(p-2)I(\nu,\nu)^{2/p-3}\\
    &\quad \cdot\bigg[2I(\nu,a)I(\nu,b)\left(I(\nu,D\nu.\nu) -\frac{p-1}{p}K(\nu,\nu,\nu)\right) \\ 
    &\qquad-I(\nu,\nu)I(\nu,a)\left(I(b,D\nu.\nu) - \frac{p-1}{p}K(\nu,\nu,b)\right)\\
    &\qquad-I(\nu,\nu)I(\nu,b)\left(I(a,D\nu.\nu)- \frac{p-1}{p}K(\nu,\nu,a)\right) \\
    &\qquad-I(\nu,\nu)I(a,b)\left(I(\nu,D\nu.\nu) - \frac{p-1}{p}K(\nu,\nu,\nu)\right)\\
    &\qquad +I(\nu,\nu)^2\left(J(\nu,a,b,D\nu.\nu)-\frac{p-1}{p}K(\nu,a,b)\right)\bigg].
\end{align*}
Putting all the terms together yields the left hand-side of the generalized Koszul formula \eqref{koszul2}, i.e.
\begin{align*}
2g^\nu(\nabla^\nu_a\nu,b)&=2(p-1)I(\nu,\nu)^{2/p-1}\left(I(D\nu.a,b)-\frac{p-1}{p}K(\nu,a,b)\right)\\
&\quad-2(p-2)I(\nu,\nu)^{2/p-2}I(\nu,b)\left(I(D\nu.a,\nu)-\frac{p-1}{p}K(\nu,\nu,a)\right).
\end{align*}
\end{proof}
As a direct consequence of the above characterization of the $\alpha$-connections as a Chern connection we obtain that these connections have an interpretation as describing energy minimizing curves:
\begin{corollary}
Let $\alpha\in (-1,1)$. Geodesic curves of the $\alpha$-connection describe locally minimizing curves of the $\frac{2}{1-\alpha}$-Energy 
\begin{align*}
E_{\frac{2}{1-\alpha}}(\mu)=\frac{1-\alpha}{2}\int_0^1 \int\Abs{\frac{\mu_t}{\mu}}^{\frac{2}{1-\alpha}}\mu\, dt.   
\end{align*}
\end{corollary}

\subsection{The $p$-root transform}
Next, we will isometrically map the space of densities to a simpler space, which will allow us to obtain explicit expressions for solutions to the geodesic equation; we call this construction, which is a direct generalization of the square-root transform for the Fisher-Rao metric, the $p$-root transform. At the same time the $p$-root transform presents an alternative way to define the $\alpha$-connection. This has been first proposed by Gibilisco and Pistone~\cite{gibilisco1998connections}, who considered this construction specifically for the  space $\Prob$ albeit with slightly different notations and a different identification of a tangent vector with a function. 
\begin{theorem}\label{thm:proot}
Endow the space $C^\infty(M)$ of smooth functions with the standard $L^p$-norm and with the trivial vector space connection $\ntr$, i.e., for two vector fields $\xi,\eta:C^\infty(M) \to C^\infty(M)$,
\[
\ntr_\xi\eta = D\eta.\xi.
\]
Let $\alpha\in (-1,1)$ and, as before, denote $p = \frac{2}{1-\alpha}$.
Define the map $\Phi_p: \Dens \to C^\infty(M)$ by
\begin{equation}\label{eq:ptransform}
\Phi_p(\mu) = \left(\frac{\mu}{\dx}\right)^{1/p}.
\end{equation}
We have:
\begin{enumerate}[label=(\alph*)]
\item The image $\Phi_p(\Dens)$  is 
the set of all positive functions in $C^{\infty}(M)$. 
\item \label{proot_FR} The mapping $\Phi_p$  
is an isometric embedding, where $\Dens$ is equipped with a multiple of the $L^p$-Fisher-Rao metric and where $C^{\infty}(M)$ is viewed as a vector space equipped with the standard $L^p$-norm.
\item \label{proot_alphaCon}The pullback of $\Phi_p^*\ntr$ coincides with $\nabla^{(\alpha)}$ up to a constant depending only on the footpoint:
\[
(\Phi_p^*\ntr)_a b|_\mu = \frac{1}{p}\left(\frac{\mu}{\dx}\right)^{\frac{1}{p}-1} \nabla^{(\alpha)}_a b|_\mu,
\qquad \mu\in \Dens, \, a,b\in \mathfrak{X}(\Dens)
\]
In particular, the geodesics of $\Phi_p^*\ntr$ and $\nabla^{(\alpha)}$ coincide.
\end{enumerate}
\end{theorem}
Note that geodesics of the trivial connection on a vector space are always straight lines; in particular, this proposition allows us to obtain geodesics of the $L^p$-Fisher-Rao metric (of the $\alpha$-connection, resp.) by pulling-back straight lines in $C^\infty(M)$ using $\Phi_p$. We will use this in Corollary~\ref{cor:Lpgeometry} below to explicitly describe the resulting formulas on $\Dens$. First we present the proof of the above theorem, which is a fairly straightforward calculation:
\begin{proof}[Proof of Theorem~\ref{thm:proot}]
The characterization of the image of $\Phi_p$ follows directly from the definition of $\Dens$. 
To show item~\ref{proot_FR} we calculate for $\mu\in \Dens$ and $a\in T_\mu\Dens$ the differential of $\Phi_p$:
 \begin{align*}
D_\mu\Phi_p(a)=\frac1p\frac{a}{\dx}\left(\frac{\mu}{\dx}\right)^{1/p-1}.
\end{align*}
Therefore the pullback of the $L^p$-norm via the embedding $\Phi_p$ is given by 
 \begin{align*}
\| D_\mu\Phi_p(a)\|_{L^p}=\frac1p\left(\int_M\left|D_\mu\Phi_p(a)\right|^{p}\ud\dx\right)^{1/p}=\frac1p\left(\int_M\Abs{\frac{a}{\mu}}^p\mu\right)^{1/p}=\frac1pF_p(\mu,a),
\end{align*}
which implies that the embedding $\Phi_p$ is indeed an isometry.

Similarly we calculate for item~\ref{proot_alphaCon}
\[
\begin{split}
(\Phi_p^*\ntr)_a b|_\mu 
    &= (T \Phi_p)^{-1} \ntr_{T \Phi_p(a)}T \Phi_p(b)|_{\Phi_p(\mu)}\\
    &= (T \Phi_p)^{-1} \left(\frac{1}{p^2}\left(\frac{\mu}{\dx}\right)^{\frac{1}{p}-1}\left(\ntr_{a/\dx}\left(\frac{\nu}{\dx}\right)^{\frac{1}{p}-1}\frac{b}{\dx}\right)_{\mu=\nu}\right)\\
    &= \frac{1}{p}\left(\ntr_{a/\dx}\left(\frac{\nu}{\dx}\right)^{\frac{1}{p}-1}\frac{b}{\dx}\right)_{\mu=\nu}\dx\\
    &= \frac{1}{p}\left(\frac{\mu}{\dx}\right)^{\frac{1}{p}-1}\left(D(b/dx).(a/dx) + \left(\frac{1}{p}-1\right)\left(\frac{\mu}{\dx}\right)^{-1}\frac{a}{\dx}\frac{b}{\dx}\right)\dx \\
    &= \frac{1}{p}\left(\frac{\mu}{\dx}\right)^{\frac{1}{p}-1}\left(Db.a + \left(\frac{1}{p}-1\right)\frac{a}{\mu}b\right) 
    = \frac{1}{p}\left(\frac{\mu}{\dx}\right)^{\frac{1}{p}-1} \nablabar^{(\alpha)}_a b|_\mu.\\
\end{split}
\]
\end{proof}
The above theorem allows us to explicitly solve for geodesics on $\Dens$, which in turn leads to a proof of metric and geodesic incompleteness of the $L^p$-Fisher-Rao metric for any $p>1$. By the equivalence of geodesics for the $\alpha$-connections and for the $L^p$-Fisher-Rao metric the formulas for geodesics also hold for the former. In the finite dimensional setting this solution formula (via the $p$-root mapping) for the $\alpha$-geodesics is known albeit without any geometric interpretation, cf.~\cite[Page 50]{ay2017information}.
\begin{corollary}[The geometry of the $L^p$-Fisher-Rao metric]\label{cor:Lpgeometry}
For any $p>1$ we have the following statements:
\begin{enumerate}[label=(\alph*)]
    \item \label{prop:convex} The space $\Dens$ equipped with the $L^p$-Fisher-Rao metric (the $\alpha$-connection resp.) is  geodesically convex and, even more, there exists an explicit formula for all minimizing geodesics: given any $\mu_0,\mu_1\in \Dens$ the unique geodesic $\mu: [0,1]\to\Dens$ connecting $\mu_0$ to  $\mu_1$ is given by
     \begin{align*}
\mu(t)=\left(t\sqrt[p]{\tfrac{\mu_1}{\dx}}+(1-t)\sqrt[p]{\tfrac{\mu_0}{\dx}}\right)^p \dx.
\end{align*}
\item  Given any $\mu_0,\mu_1\in \Dens$ the geodesic distance of the $L^p$-Fisher-Rao metric is given by
\begin{equation*}
d(\mu_0,\mu_1) = \left( \int_M \left| \sqrt[p]{\tfrac{\mu_1}{\dx}}-\sqrt[p]{\tfrac{\mu_0}{\dx}}  \right| \dx\right)^{1/p}
\end{equation*}
In particular, the geodesic distance of the $L^p$-Fisher-Rao metric on $\Dens$ is non-degenerate.
    \item  For any initial conditions $\mu_0\in \Dens$ and $a\in T_\mu \Dens$ the unique $L^p$-Fisher-Rao geodesic ($\alpha$-connection geodesic, resp.) $\mu:[0,T)\to \Dens$
    defined on its maximal interval of existence $[0,T)$ is given by
    \begin{equation*}
    \mu(t)=\left(\sqrt[p]{\tfrac{\mu_0}{\dx}}+
    t\frac{a}{\dx}\left(\frac{\mu}{\dx}\right)^{1/p-1}\right)^p \dx.
    \end{equation*}
    The geodesic $\mu(t)$ exists for all time $t$, i.e., $T=\infty$, if and only if $\frac{a}{\dx}(x)\geq0$ for all $x\in M$.
    Thus the space $\Dens$ equipped with the $L^p$-Fisher-Rao metric is  geodesically incomplete since the solution to the geodesic equation~\eqref{eq:geodesicDens} leaves the space in finite time for any initial condition with $\frac{a}{\dx}(x)<0$ for some $x$. 
    \item \label{prop:metriccomplete} The space $\Dens$ equipped with the geodesic distance of the $L^p$-Fisher-Rao metric is  metrically incomplete.
    \item The metric completion of the space $\Dens$ with respect to the geodesic distance of the $L^p$-Fisher-Rao metric is the space of all non-negative $L^1$-densities:
    \begin{equation*}
       \operatorname{Dens}_{L^1}(M)=\left\{\mu: \frac{\mu}{\dx}\in L^1(M),  \frac{\mu}{\dx} \geq 0 \text{ a.e.}\right\} 
    \end{equation*}
    \end{enumerate}
\end{corollary}
\begin{proof}
Statements~\ref{prop:convex}--\ref{prop:metriccomplete} follow directly from the isometry of Theorem~\ref{thm:proot}, the fact that geodesics  on the vector space 
$(C^{\infty}(M),L^p)$ are straight lines and the characterization of the image of $\Phi_p$ as an open, convex subset of $C^{\infty}(M)$. To see the statement regarding the metric completion we observe that the metric completion of the image is exactly the set of a.e.~non-negative $L^p$-functions and thus the statement on the metric completion follows by applying $\Phi_p^{-1}$. 
\end{proof}

\section{The $L^p$-Fisher-Rao metric and $\alpha$-connections on the space of probability densities}\label{sec:prob}

The $L^p$-Fisher-Rao metric $F_p$ and the $\alpha$-divergence $D^{\alpha}$ define, via restriction, corresponding objects on $\Prob$, which we study in this section. 
In particular, we will see that $\Prob$ equipped with the $L^p$-Fisher-Rao metric corresponds geometrically to an infinite dimensional $L^p$-sphere. In addition we will see that the equivalence to the $\alpha$-connection, that has been established for the space of all densities in the previous section, does not hold on the space of probability densities. Consequently we obtain three different notions of $p$-geodesics on this space:
\begin{enumerate}
    \item geodesics of the restriction of the $L^p$-Fisher-Rao metric to $\Prob$; 
    \item \label{item:alphageo} geodesics of the $\alpha$-connections on $\Prob$;
    \item \label{item:projgeo} projections of $L^p$-Fisher-Rao geodesic curves (or equivalently, the $\alpha$-connection ones) on $\Dens$.
\end{enumerate}
In addition, if we allow to leave the space of probability densities, we obtain a fourth notion:
\begin{enumerate}[resume]
    \item $L^p$-Fisher-Rao geodesics in $\Dens$.
        In analogy to the $L^2$ case, the induced geodesic distance between probability densities defines an $L^p$ version of the Hellinger distance.
\end{enumerate}
We will show that \eqref{item:alphageo} and \eqref{item:projgeo} coincide, thereby providing an explicit formula for $\alpha$-geodesics on $\Prob$. For a graphic summary of these constructions we refer to Figure~\ref{fig:p_root}. 
In the next section we will compare the remaining three notions of geodesics numerically. 

\subsection{The Amari-\u{C}encov $\alpha$-connections on $\Prob$}
The restriction of the $\alpha$-divergences $D^{\alpha}$ to the space $\Prob$ induces again a family of $\alpha$-connections, which we will denote by  $\nablabar^{(\alpha)}$. Note, that this connection is not simply the restriction of the $\alpha$-connections on $\Dens$, which is the reason for choosing a different notation for it.
We start by deriving an explicit formula for the $\alpha$-connections on $\Prob$: 
\begin{lemma}
 For any $\alpha\in (-1,1)$ the $\alpha$-connections $\nablabar^{(\alpha)}$ on $\Prob$ are given by
    \begin{equation}\label{alpha-prob}
    \nablabar^{(\alpha)}_a b = Db.a - \frac{1}{p^*} \left(\frac{a}{\mu}b -\left(\int_M \frac{a}{\mu}\frac{b}{\mu} \mu\right)\mu\right).
    \end{equation}
Thus, the connection  $\nablabar^{(\alpha)}_a b$ on $\Prob$ is the projection of $\nabla^{(\alpha)}_a b$ with respect to the Fisher-Rao metric $\GFR$.    
\end{lemma}

For finite sample spaces this result is well-known (e.g., \cite[Section~2.5.2]{ay2017information});
in infinite dimensions formula \eqref{alpha-prob} agrees with the formula (22) in \cite{lenells2014amari}, under the identification of $\Prob = \Diff/\SDiff$.

\begin{proof}
To derive the formula for the $\alpha$-connection $\nablabar^{(\alpha)}$ we calculate the second derivative of the restriction of $D^{\alpha}$, which is given again by formula~\eqref{eq:a_conn_def} with the only difference being that $a,b,c\in T_\mu\Prob$. Thus we have determined $\nablabar^{(\alpha)}$ up to a function in the $\GFR$ orthogonal complement of $T_\mu\Prob$, which are exactly the constant multiples of $\mu$. Thus the formula follows by ensuring that  $\nablabar^{(\alpha)}_a b\in \Prob$. 
This argument also proves that  $\nablabar^{(\alpha)}_a b$ is the Fisher-Rao projection of $\nabla^{(\alpha)}_a b$.  
\end{proof}
\begin{theorem}\label{eq:alpha_conn_geo_Prob}
A path $\mu:[0,1]\to \Prob$ is a geodesic with respect to $\nabla^{(\alpha)}$ if
\[
\mu_{tt} - \frac{1}{p^*} \mu^{-1}\mu_t^2
=
- \frac{1}{p^*} \left(\int \left(\frac{\mu_t}{\mu}\right)^{2}\mu\right) \mu.
\]
For any $k>\operatorname{dim}(M)/2$ the geodesic equations are locally wellposed on the space of Sobolev probability densities $\Probk$, i.e., given initial conditions $\mu(0)\in \Probk$, $\mu_t(0)\in T_{\mu(0)}\Probk$ there exists an unique solution to equation~\eqref{eq:geo_alphacon} defined on a maximal interval of existence $[0,T)$. 
The maximal interval of existence is uniform in the Sobolev order $k$ and thus the local wellposedness continues to hold in the limit, i.e., on the space of smooth, probability densities $\Prob$.
\end{theorem}
\begin{proof}
The proof of the local wellposedness follows exactly as in Theorem~\ref{thm:geo_alpha_dens}.     
\end{proof}

\subsection{The $L^p$-Fisher-Rao metric on $\Prob$}
Next, we study the restriction of the $L^p$-Fisher-Rao metric to the space $\Prob$.
\begin{remark}[\u{C}encov's theorem]\label{rem:uniqueness}
Note that Lemma~\ref{lem:invariance} on the invariance of the $L^p$-Fisher-Rao metric continues to hold on the space $\Prob$. For the Riemannian case and $\operatorname{dim}(M)>1$ \u{C}encov's theorem states that the Fisher-Rao metric is the only Riemannian metric on $\Prob$ that is invariant under the action of the diffeomorphism group $\Diff$, cf.~\cite{cencov2000statistical,ay2015information,bauer2016uniqueness}. In the Finslerian case there is a significant amount of additional flexibility, and one can indeed construct metrics beyond the $L^p$-Fisher-Rao metric that satisfy this property. In future work it would be interesting to obtain a complete characterization of all such Finsler metrics. 
\end{remark}

We start by computing the geodesic equation of the (restriction) of the $L^p$-Fisher-Rao metric $F_p$ on $\Prob$:
\begin{theorem}[Geodesic equation on $\Prob$]\label{thm:LpFR_geo_prob}
For any $p\in(1,\infty)$, the geodesic equation of the $L^p$-Fisher-Rao metric on the space of densities $\Prob$ is given by
\begin{align}
\Abs{\frac{\mu_t}{\mu}}^{p-2}\frac{d }{dt}\left(\frac{\mu_t}{\mu}\right)+\frac{1}{p}\Abs{\frac{\mu_t}{\mu}}^p=C(t)
\end{align}
where $C(t)$ is  a constant depending only on time $t$, that is chosen such that $\int_M \mu(t) =1$.

This equation coincides with the geodesic equation of the $\alpha$-connection if and only if $p=2$ ($\alpha=0$, resp.).
\end{theorem}
\begin{remark}[Existence of solutions]
In the previous section we showed that the geodesic equation of the $\alpha$-connections is locally wellposed on the space $\Prob$. One would be tempted to expect a similar result for the geodesic equation of the $L^p$-Fisher-Rao metric; recall that this statement was true on the space $\Dens$. 
It turns out that the above equation is analytically much worse-behaved: the problem arises from the vanishing of the quantity $\frac{\mu_t}{\mu}$ which leads to singularities of the geodesic equation. As a consequence we conjecture that the geodesic equation does not admit any classical solutions. This behavior can also be observed in the numerical simulations (Figure~\ref{fig:geodesics}), where the obtained (approximate) solutions show a singular behavior.        
\end{remark}
\begin{proof}[Proof of Theorem~\ref{thm:LpFR_geo_prob}]
To derive this equation, we proceed as for the geodesic equation on the space $\Dens$. We then obtain again 
\begin{equation*}
\begin{aligned}
\delta E_p(\mu)(\delta \mu)
&=-\frac{1}{p}\int_0 ^1 \int  \left( p\frac{d}{dt}\left(\Abs{\frac{\mu_{t}}{\mu}}^{p-2}\frac{\mu_t}{\mu}\right)+(p-1)\Abs{\frac{\mu_t}{\mu}}^p \right) \  \delta \mu\ \ud \dx\, dt:=-\frac{1}{p} \int_0 ^1 \int  \Psi \  \delta \mu\ \ud \dx\, dt
\end{aligned}
\end{equation*}
for the variation of the $p$-Energy with the only difference being that $\delta \mu$ now has to integrate to zero.
Thus we do not get that $\Psi=0$ as we had on the space $\Dens$, but   only that $\Psi$ has to be orthogonal to all such $\delta \mu$. This is equivalent to $\Psi$ being a constant for each fixed time $t$, which is determined by the condition that $\int_M \mu(t)=1$.
\end{proof}
The above result suggests that the equivalence between the $\alpha$-connection and the Chern-connection of the $L^p$-Fisher-Rao metric cannot hold in this setting. We will make this formal in the following theorem:
\begin{theorem}[The Chern connection on $\Prob$]\label{thm:chern_prob}
    For a vector field $\nu$ on $\Prob$ the Chern connection is given by, for all $a\in T_\mu\Prob$,
    \begin{equation}\label{chern-prob}
    \nablabar^\nu_a\nu=D \nu.a-\frac{p-1}{p}\frac{a}{\mu}\frac{\nu}{\mu}\mu + k_1\Abs{\frac{\nu}{\mu}}^{-p}\frac{a}{\mu}\nu + k_2\Abs{\frac{\nu}{\mu}}^{2-p}\mu,
    \end{equation}
    with the constants
    \begin{align*}
    &k_1(\nu):=-\frac{(p-1)(p-2)}{2p}\left(\int\Abs{\frac{\nu}{\mu}}^2\mu\right) / \left(\int\Abs{\frac{\nu}{\mu}}^{2-p}\mu\right)\\
    &k_2(\nu,a):=\frac{p-1}{p}\left(\int\frac{a}{\mu}\frac{\nu}{\mu}\mu\right)/\left(\int \Abs{\frac{\nu}{\mu}}^{2-p}\mu\right) + \frac{(p-1)(p-2)}{2p}\left(\int\Abs{\frac{\nu}{\mu}}^{-p}\frac{a}{\mu}\nu\int \Abs{\frac{\nu}{\mu}}^{2}\mu \right)/\left(\int \Abs{\frac{\nu}{\mu}}^{2-p}\mu\right)^2.
    \end{align*}
\end{theorem}
\begin{remark}
As any vector field $\nu\in T_\mu\Prob$ has zeros the above formula has to be taken with caution and should be understood formally only. 
\end{remark}

\begin{remark}\label{rem:chern-prob}
In particular, when all entries are the same, the Chern connection on $\Prob$ is the orthogonal projection of the $\alpha$-connection $\nabla^{(\alpha)}$ on $\Dens$, for $\alpha=1-\frac{2}{p}$, with respect to $g^\nu$, the Riemannian metric \eqref{riem-metric} induced by the $L^p$-Fisher-Rao metric
    \begin{equation}\label{chern-prob-2}
    \nablabar^\nu_\nu\nu= \mathrm{Proj}^{\nu}\left(\nabla^{(\alpha)}_\nu\nu\right) = D\nu.\nu -\frac{1}{p^*}\frac{\nu}{\mu}\frac{\nu}{\mu}\mu + k\Abs{\frac{\nu}{\mu}}^{2-p}\mu
    \end{equation}
    where $p^*$ is the Hölder conjugate of $p$ and 
    $$k(\nu):=\frac{1}{p^*}\left(\int\Abs{\frac{\nu}{\mu}}^2\mu\right) / \left(\int\Abs{\frac{\nu}{\mu}}^{2-p}\mu\right).$$
    Indeed, the correction term $k\Abs{\frac{\nu}{\mu}}^{2-p}\mu$ is orthogonal to $T\Prob$ and makes the integral zero.
\end{remark}
\begin{proof}[Proof of Theorem~\ref{thm:chern_prob}]
We start by noticing that, since $D\nu(a)$ integrates to zero, the integral of the right hand-side of \eqref{chern-prob} is zero and so it defines a tangent vector of $\Prob$. The formula \eqref{chern-prob} defines the Chern connection if and only if it verifies the generalized Koszul formula \eqref{koszul2}.
Letting $\alpha=1-\frac{2}{p}$ and $\nabla^{(\alpha)}$ be the corresponding $\alpha$-connection on $\Dens$, we can decompose the candidate for the Chern connection as
$$\nablabar^\nu_a\nu=\nabla^{(\alpha)}_a\nu + k_1\Abs{\frac{\nu}{\mu}}^{-p}\frac{a}{\mu}\nu + k_2\Abs{\frac{\nu}{\mu}}^{2-p}\mu.$$
Since $\nabla^{(\alpha)}$ is the Chern connection on $\Dens$ for this choice of $\alpha$, the candidate \eqref{chern-prob} verifies the generalized Koszul formula if and only if
\begin{equation}\label{koszul3}
\begin{aligned}
0=&2g^\nu\left(k_1(\nu)\Abs{\frac{\nu}{\mu}}^{-p}\frac{a}{\mu}\nu+ k_2(\nu,a)\Abs{\frac{\nu}{\mu}}^{2-p}\mu,b\right)  + 2C^\nu\left(k_1(\nu)\Abs{\frac{\nu}{\mu}}^{-p}\frac{\nu}{\mu}\nu +k_2(\nu,\nu)\Abs{\frac{\nu}{\mu}}^{2-p}\mu , b, a\right)\\
=&2k_1(\nu)g^\nu\left(\Abs{\frac{\nu}{\mu}}^{-p}\frac{a}{\mu}\nu, b\right)  + 2k_2(\nu,a)g^\nu\left(\Abs{\frac{\nu}{\mu}}^{2-p}\mu,b\right) + 2(k_1(\nu)+k_2(\nu,\nu)) C^\nu\left(\Abs{\frac{\nu}{\mu}}^{2-p}\mu,b,a\right).
\end{aligned}
\end{equation}
Noticing that, for all $b\in T_\mu\Prob$, $I(|\frac{\nu}{\mu}|^{2-p}\mu, b)=\int b = 0$, we see from \eqref{riem-metric} that 
$$g^\nu\left(\Abs{\frac{\nu}{\mu}}^{2-p}\mu, b\right)=0.$$ 
This also means that all terms in the Cartan tensor \eqref{cartan-tensor} but one vanish, leaving
\begin{align*}
2C^\nu\left(\Abs{\frac{\nu}{\mu}}^{2-p}\mu,b,a\right)&=(p-1)(p-2)I(\nu,\nu)^{\frac{2}{p}-1}J\left(\nu,\Abs{\frac{\nu}{\mu}}^{2-p}\mu,b,a\right)\\
&=(p-1)(p-2)I(\nu,\nu)^{\frac{2}{p}-1}\int \left(\frac{\nu}{\mu}\right)^{-1}\frac{a}{\mu}\frac{b}{\mu}\mu
\end{align*}
Finally there remains to compute
\begin{align*}
g^\nu\left(\Abs{\frac{\mu}{\nu}}^p\frac{a}{\mu}\nu, b\right)&=(p-1)I(\nu,\nu)^{\frac{2}{p}-1}\int \Abs{\frac{\nu}{\mu}}^{p-2}\Abs{\frac{\nu}{\mu}}^{-p}\frac{a}{\mu}\frac{\nu}{\mu}\frac{b}{\mu}\mu\\
&-(p-2)I(\nu,\nu)^{\frac{2}{p}-2}I(\nu,b)\int \Abs{\frac{\nu}{\mu}}^{p-2}\frac{\nu}{\mu}\Abs{\frac{\nu}{\mu}}^{-p}\frac{a}{\mu}\frac{\nu}{\mu}\mu\\
&=(p-1)I(\nu,\nu)^{\frac{2}{p}-1}\int \left(\frac{\nu}{\mu}\right)^{-1}\frac{a}{\mu}\frac{b}{\mu}\mu.
\end{align*}
Putting all these together, and noticing that $k_2(\nu,\nu)=-\frac{p}{p-2}k_1(\nu)$, we obtain 
\begin{align*}
&2g^\nu\left(k_1(\nu)\Abs{\frac{\nu}{\mu}}^{-p}\frac{a}{\mu}\nu+ k_2(\nu,a)\Abs{\frac{\nu}{\mu}}^{2-p}\mu,b\right)  + 2C^\nu\left(k_1(\nu)\Abs{\frac{\nu}{\mu}}^{-p}\frac{\nu}{\mu}\nu +k_2(\nu,\nu)\Abs{\frac{\nu}{\mu}}^{2-p}\mu , b, a\right)\\
&=\left(pk_1(\nu)+(p-2)k_2(\nu)\right)(p-1)I(\nu,\nu)^{\frac{2}{p}-1}\int \left(\frac{\nu}{\mu}\right)^{-1}\frac{a}{\mu}\frac{b}{\mu}\mu\\
&=0,
\end{align*}
and so condition \eqref{koszul3} is satisfied. 
\end{proof}

\subsection{The $p$-root transform on $\Prob$}
In the previous section we have seen that the $\alpha$-connection and the $L^p$-Fisher-Rao metric induce different geodesics on the space $\Prob$. In this section we will investigate the geometric reasons behind this, by connecting both of these objects to the $p$-root transform.
In order to state this result we will need to define an appropriate connection on the sphere 
\[
S_p := \{ f\in C^{\infty}(M) ~:~ \|f\|_{L^p} = 1\},
\]
as the image of $\Prob$ under $\Phi_p$ is in this set.
To this end, we define:
\begin{definition}[$p$-projection and $p$-connection]
The $p$-projection map $\pi^p: TC^\infty|_{S_p} \to TS_p$ is defined by
\[
\pi^p_f(\xi) = \xi - \left(\int_M \xi f |f|^{p-2}\,\ud\dx\right)f, \qquad f\in S_p,\, \xi\in TC^\infty|_{S_p}.
\]
The induced $p$-connection on $S_p$ is defined by
\[
\nabla^p_\xi\eta = \pi^p \left( \ntr_\xi\eta\right).
\]
\end{definition}
Note, that $\pi^p$ is the projection with respect to the splitting $T_fC^\infty = T_f S_p \oplus \text{span}\{f\}$.
The geodesic equation $\nabla^p_{\dot\gamma}\dot\gamma = 0$ can therefore be written as:
\begin{equation}\label{eq:geod_nabla_p}
\begin{cases}
    \ddot\gamma \parallel \gamma \\
    \int_M \gamma^p \,\ud\dx = 1
\end{cases}
\end{equation}
Note that from a metric point of view, this splitting is natural since $f\in T_fC^\infty$ is the unique direction from which straight lines (i.e., geodesics in $C^\infty$) emanating from $f$ gets the fastest away from $S_p$ with respect to the $L^p$ norm (since for $p\in (1,\infty)$ the space $L^p$ is strictly convex).
Similarly, $\pi^p_f(\xi)$ satisfies $\|\xi - \pi^p_f(\xi)\|_{L^p} = \dist_{L^p} (\xi, T_fS_p)$.
For a more general viewpoint on projections on a sphere in uniformly convex Banach spaces whose dual is also uniformly convex, see \cite{gibilisco1999connections} and \cite[Prop.~2]{gibilisco2020p}.

We are now able to formulate the analogous statement of Theorem~\ref{thm:proot}, which will demonstrate the geometric differences between the $\alpha$-connections and the $L^p$-Fisher-Rao metric: 
\begin{theorem}\label{thm:proot_Prob}
Let $\alpha\in (-1,1)$ and, as before, denote $p = \frac{2}{1-\alpha}$.
Consider the restriction of the map $\Phi_p$, as defined in
\eqref{eq:ptransform}, to the space $\Prob$.
We have:
\begin{enumerate}[label=(\alph*)]
\item The image $\Phi_p(\Prob)$  is 
the set of all positive functions in the $L^p$-sphere $S_p$. 
\item \label{proot_FR_Prob} The mapping $\Phi_p$  
is an isometric embedding, where $\Prob$ is equipped with a multiple of the $L^p$-Fisher-Rao metric and where $S_p$ is equipped with the restriction of the standard $L^p$-norm.
\item \label{proot_alphaCon_Prob}
The pullback of $\Phi_p^*\nabla^p$ to $\Prob$ coincides with the connection $\overline{\nabla}^{(\alpha)}$ up to a constant depending only on the footpoint:
\[
(\Phi_p^*\nabla^p)_a b|_\mu = \frac{1}{p}\left(\frac{\mu}{\dx}\right)^{\frac{1}{p}-1} \overline{\nabla}^{(\alpha)}_a b|_\mu,
\qquad \mu\in \Prob, \, a,b\in \mathfrak{X}(\Prob).
\]
In particular, the geodesics of $\Phi_p^*\nabla^p$ and $\overline{\nabla}^{(\alpha)}$ coincide.
\end{enumerate}
\end{theorem}
\begin{proof}
The proof follows by the same calculation as the proof of Theorem~\ref{thm:proot}.
\end{proof}
On $S_p$, geodesics are no longer straight lines, and we do not have an explicit solution for the geodesic equations of either the $\alpha$-connection or the $L^p$-Fisher-Rao metric. 
However, by projecting straight lines on the sphere and rescaling time, one can obtain geodesics for the $\alpha$-connection (cf. \cite[Section~2.5.2]{ay2015information} where this result has been shown in the finite dimensional situation):

\begin{theorem}\label{prop:alpha-geod-prob}
    Let $f\in S_p$ and $\xi\in T_f S_p$.
    Let $I\subset \R$ be an interval containing $0$, and let $\tau:I\to \R$ satisfy the ODE
    \[
    \begin{split}
    \ddot\tau(t) &= 2\frac{\int_M |f + \tau(t) \xi|^{p-2}(f + \tau(t) \xi)\xi\,\ud\dx}{\int_M|f + \tau(t) \xi|^p\,\ud\dx}\dot\tau(t)^2 \\
    \tau(0) &= 0 \\
    \dot\tau(0) &= 1
    \end{split}
    \]
    Then $\gamma:I\to S_p$ defined by
    \[
    \gamma(t) = \frac{f + \tau(t)\xi}{\|f+\tau(t)\xi\|_{L^p}}
    \]
    is a geodesic of $\nabla^p$, with initial condition $\gamma(0) = f$, $\dot\gamma(0) = \xi$.
    
    A boundary value problem between $f,g\in S_p$ can be similarly addressed by putting $\xi = g-f$ and $I=[0,1]$, and replacing the initial conditions for $\tau$ by the boundary conditions $\tau(0)=0$, $\tau(1)=1$.

    Geodesics of $\overline{\nabla}^{(\alpha)}$ are obtained by pulling back these geodesics using $\Phi_p$.
    They all cease to exist (i.e., leave the space $\Prob$) after finite time.
    Since the geodesic equation is locally well-posed (Proposition~\ref{eq:alpha_conn_geo_Prob}), this procedure induces all the $\alpha$-connection geodesics, i.e., the exponential map of $\overline{\nabla}^{(\alpha)}$.
\end{theorem}

\begin{proof}
    Using \eqref{eq:geod_nabla_p}, we need to show that $\ddot\gamma \parallel\gamma$; all the other assumptions are satisfied by construction.
    We have
    \[
    \ddot\gamma = \ddot\tau(t)\|f+\tau(t)\xi\|_{L^p}^{-1}\xi + 2\dot\tau(t)\frac{d}{dt}\|f+\tau(t)\xi\|_{L^p}^{-1}\xi + \frac{d^2}{dt^2}\|f+\tau(t)\xi\|_{L^p}^{-1}(f + \tau(t)\xi).
    \]
    The last addend is clearly parallel to $\gamma$.
    Hence it is sufficient to require that
    \[
   \ddot\tau(t)\|f+\tau(t)\xi\|_{L^p}^{-1} + 2\dot\tau(t)\frac{d}{dt}\|f+\tau(t)\xi\|_{L^p}^{-1} = 0,
    \]
    which is equivalent to the wanted ODE.

    In order to prove that the pullback of the solutions leaves $\Prob$ after a finite time, we need to show that $\gamma(t)$ stops being positive, i.e., that for some $t>0$, $f(x) + \tau(t)\xi(x) \le 0$ for some $x\in M$.
    From the equivariance under the action of $\Diff$, cf. Remark~\ref{rem:uniqueness}, it is sufficient to consider the case $f\equiv 1$ (which corresponds to $\mu(0)=\dx$).
    In this case $\xi$ is a non-zero function satisfying $\int_M \xi \,\dx= 0$, and thus in particular $\xi(x) <0$ for some $x$.
    Therefore, in order to prove that $1 + \tau(t)\xi(x) \le 0$ for some $t$, it is sufficient to prove that $\tau$ is unbounded as $t\to \infty$.
    Note that we can write the equation for $\tau$ as
    \begin{equation}\label{eq:tau}
    \ddot\tau(t) = 2\left(1-\frac{1}{\int_M|1 + \tau(t) \xi|^p\,\ud\dx}\right)\frac{\dot\tau(t)^2}{\tau}.
    \end{equation}
    Now, since $s\mapsto 1 + s\xi$ is a tangent line to the unit sphere at $f=1$ in the strictly convex space $L^p$, it follows that $\|1+s\xi\|_{L^p} \ge 1$, and equality holds if and only if $s=0$.
    Thus, the term in the parentheses in \eqref{eq:tau} is non-negative, and vanishes if and only if $\tau(t) = 0$.
    Since we also have that $\tau(0)=0$ and $\dot\tau(0)=1$, it follows that $\ddot\tau(t) > 0$ for $t\in (0,t_0)$ for some $t_0$ small enough, and thus for any positive $t$.
    It follows therefore that $\tau > t$ for all $t>0$, and in particular, it is unbounded.
\end{proof}

\begin{remark}
In fact, the estimate $\tau>t$ implies that $1+\tau(t) \xi$ hits zero at some point for the first time at $t_* < \frac{1}{-\min \xi}$.
Pulling back to $\Prob$, we obtain that a geodesic from $\dx$ with initial condition $a\in T_\dx\Prob$ blows up at time 
\begin{equation}\label{eq:blowup}
t_* < \frac{p}{-\min(a/\dx)}.
\end{equation}
In principle, better estimates on the blowup can be obtained by more careful analysis of \eqref{eq:tau}.
The estimate \eqref{eq:blowup} is exactly the estimate obtained in \cite[Formula (78)]{kogelbauer2020global} (there, the parameter $a$ is equivalent to $-1-\frac{2}{p}$ in our notation).
\end{remark}

\begin{example}[Fisher-Rao geodesics]
For the case $p=2$, assuming that $\xi$ is a unit vector (which is, by definition, perpendicular to $f$), we obtain that the ODE takes a simpler form
\[
\ddot\tau = \frac{2\tau}{1+\tau^2}\dot\tau^2,
\]
whose solution is $\tau(t) = \tan t$, yielding the known solution of the Fisher--Rao geodesics \cite[Remark~4.4]{khesin2013geometry}. 
\end{example}

\section{Summary of relations to known PDEs and metrics on diffeomorphism groups}\label{sec:PDEs}
We now summarize how the $L^p$-Fisher-Rao metric relates to (degenerate) right-invariant Finsler metric on the group of diffeomorphisms, in a similar spirit as in~\cite{khesin2013geometry} who studied this for the $L^2$-case.
Furthermore, we will see how the geodesics equations described in this paper relate to other previously studied equations in hydrodynamics and mathematical physics:
\begin{itemize}
\item On the diffeomorphism group of a closed manifold $M$ one can consider the family of, right-invariant (degenerate) $\dot W^{1,p}$-Finsler metrics of the form  
\begin{equation*}
\tilde F_p(\varphi,X\circ\varphi)=\left(\int_M |\operatorname{div}(X)|^p \ud\dx\right)^{1/p},\qquad X\in T_{\varphi}\Diff.
\end{equation*}
These metrics were useful for proving that the diameter of $\Diff$ with respect to some critical Sobolev Riemannian metrics is infinite \cite{bauer2021can}.
Note that the kernel of the Finsler metric $\tilde F_p$ consists exactly of all divergence free vector fields, and thus $\tilde F_p$ is only a  ``true'' Finsler metric on the quotient space $\Diff/\SDiff$. 
The relation to the $L^p$-Fisher-Rao metric, as studied in the present article, becomes clear by considering the mapping $\varphi\mapsto \operatorname{Jac}(\varphi)\lambda$, which gives rise to an isometry 
\begin{equation*}
 (\Diff/\SDiff,\tilde F_p)\to(\Prob,F_p).  
\end{equation*} Note, that this result is a direct generalization of the case $p=2$ treated in Khesin et al.~\cite{khesin2013geometry}. 
For this case Modin~\cite{modin2015generalized} constructed an extension of the metric $\tilde F_2$ to obtain a non-degenerate, right invariant Riemannian metric on the full group of diffeomorphisms $\Diff$, that still descends to the Fisher-Rao metric $F_2$ on $\Prob$. In future work it would be interesting to consider a similar extension for the case $p\neq 2$.
\item Similarly, the $\alpha$-connections on $\operatorname{Prob}(M)$ can be pulled back to $\Diff/\SDiff$; the corresponding geodesic equation (which is equivalent to the one in Theorem~\ref{eq:alpha_conn_geo_Prob}) was first considered in \cite{lenells2014amari}.
Theorem~\ref{prop:alpha-geod-prob} shows their integrability and finite-time blowup.
\item For the special case $M=S^1$, where the group of volume preserving diffeomorphisms is given by the group of rotations $\operatorname{Rot}(S^1)$, the $\alpha$-connections on $\operatorname{Prob}(S^1)$ can thus be pulled back to $\operatorname{Diff}(S^1)/\operatorname{Rot}(S^1)$, where the associated geodesic equation, when presented on the Lie algebra, is the generalized periodic inviscid Proudman--Johnson equation 
\[
u_{txx}+(2-\alpha)u_xu_{xx} + uu_{xxx} = 0,
\]
as was first shown in \cite{lenells2014amari}.
See \cite{sarria2013blow,kogelbauer2020global} and the references therein for analysis of this equation, also beyond the range $\alpha \in (-1,1)$.
\item Similarly, the $L^p$-Fisher--Rao metric on $\operatorname{Prob}(S^1)$ can be considered as a Finsler metric on $\operatorname{Diff}(S^1)/\operatorname{Rot}(S^1)$.
The resulting geodesic equation is the periodic $r$-Hunter--Saxton equation for $r=1/p$, as considered in \cite{cotter2020r,bauer2022geometric}.
As shown in this paper, this is not the same equation as the one of the $\alpha$-connections on $\operatorname{Prob}(S^1)$ (i.e., the generalized periodic invicid Proudman--Johnson equation), unlike what we erroneously stated in \cite{bauer2022geometric}. 
\item For $M = \R$, the geodesic equations of $\alpha$-connections (equiv., of the $L^p$-Fisher-Rao metric) on $\operatorname{Dens}(\R)$ can be considered as equations of an appropriate subgroup of $\operatorname{Diff}(\R)$, defined in \cite{bauer2022geometric}.
The resulting equation is the generalized non-periodic invicid Proudman--Johnson equation, or equivalently, the non-periodic $r$-Hunter--Saxton equation (for $r = 1/p$) \cite{cotter2020r}.
Moreover, the metric $\tilde F_p$ described above on this subgroup of $\operatorname{Diff}(\R)$ yields a similar isometry to $(\Dens(\R),F_p)$, as follows from \cite{bauer2022geometric}.
It is interesting whether $(\Dens(\R),F_p)$ can be similarly interpreted on compact manifolds as well, maybe in a similar way to the "simple unbalanced optimal transport" extension, introduced recently in \cite{khesin2023simple}.
\end{itemize}

\section{A numerical comparison of geodesics on $\Dens$ and $\Prob$}\label{sec:numerics}
In this section we aim to numerically compare the different notions of geodesics that we have encountered in this article. Given two probability densities 
we consider three notions of geodesics:
\begin{enumerate}
    \item The geodesic for the $L^p$-Fisher-Rao metric and the $\alpha$-connection on $\Dens$, which is simply obtained as the pullback by the $p$-root transform $\Phi_p$ of the straight line in $L^p$. This geodesic leaves the space $\Prob$.
    \item The geodesic for the $\alpha$-connection on $\Prob$, which is the pullback by the $p$-root transform of the projection of the straight line on the $L^p$ sphere, as described in Theorem \ref{prop:alpha-geod-prob}.
    \item The geodesic for the $L^p$-Fisher-Rao geodesic on $\Prob$, which is the pullback by the $p$-root transform of the geodesic of the $L^p$-metric restricted to the $L^p$-sphere.
\end{enumerate}
Specifically we consider the example of probability densities on the one-dimensional base space $M=[0,1]$. Note, that we have an explicit formula for the first two notions of geodesics (geodesics on $\Dens$ and $\alpha$-connection geodesics on $\Prob$), but that the calculation of the $L^p$-Fisher-Rao geodesic between two probability distributions $\mu_0$ and $\mu_1$ requires us to solve an optimization problem: the geodesic boundary value problem on the $L^p$-sphere. Namely, we minimize the $p$-energy for the $L^p$ metric on smooth functions
\begin{equation}\label{p-energy}
E_p(f)=\frac1p\int_0^1\int |f_t|^p\,\ud\dx\, dt,   
\end{equation}
where $f: [0,1]\to C^\infty(M)$ is a path constrained to belong to the $L^p$-sphere, such that $f(0)=\Phi_p(\mu_0$), $f(1)=\Phi_p(\mu_1)$ and $f_t$ denotes its time derivative. This is equivalent to minimizing the length functional, as explained in the proof of Theorem \ref{thm:geo_lp_dens}. We then obtain the wanted geodesic $\mu:[0,1]\rightarrow \Prob$ by applying $\Phi_p^{-1}$.
 
In Figure~\ref{fig:geodesics} we show the three types of geodesics obtained for different values of $p$ ($p=2,3,5, 10$ from top to bottom), and the corresponding values of $\alpha=1-2/p$. The constrained minimization of \eqref{p-energy} was performed in Python using the Sequential Least Squares Programming (SLSQP) method provided by the Scipy minimization solver, with a discretization of $30$ time points and $100$ sampling points, in a straightforward implementation that was not aimed for computational efficiency. As expected, the $L^p$-Fisher-Rao metric and the $\alpha$-connection yield different geodesics on $\Prob$, except for the special case $p=2$ corresponding to the Fisher-Rao metric and its Levi-Civita connection.

\begin{figure}
    \centering
    \includegraphics[width=0.28\linewidth]{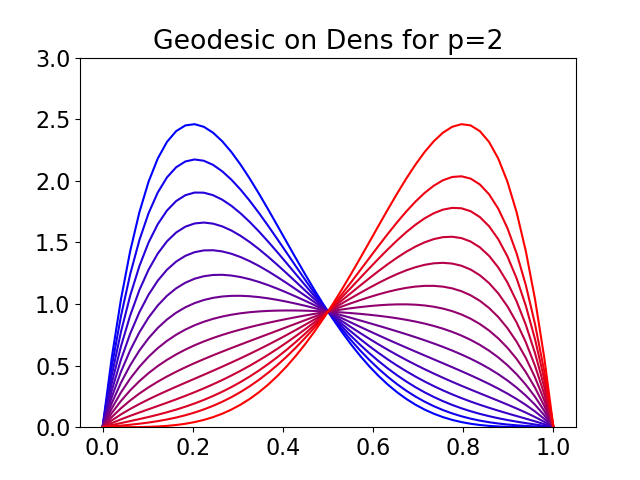}
    \includegraphics[width=0.28\linewidth]{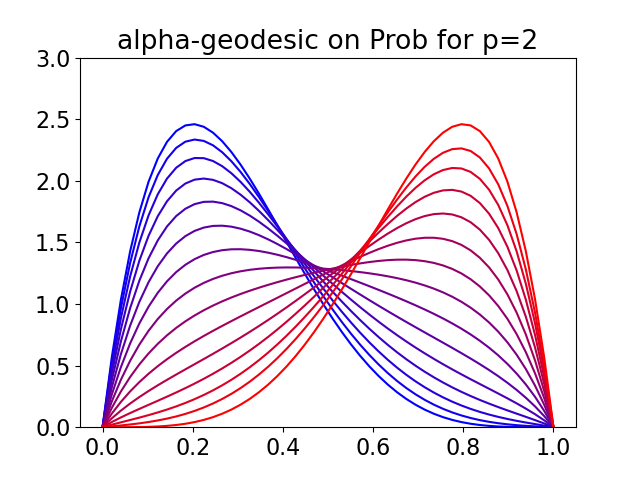}
    \includegraphics[width=0.28\linewidth]{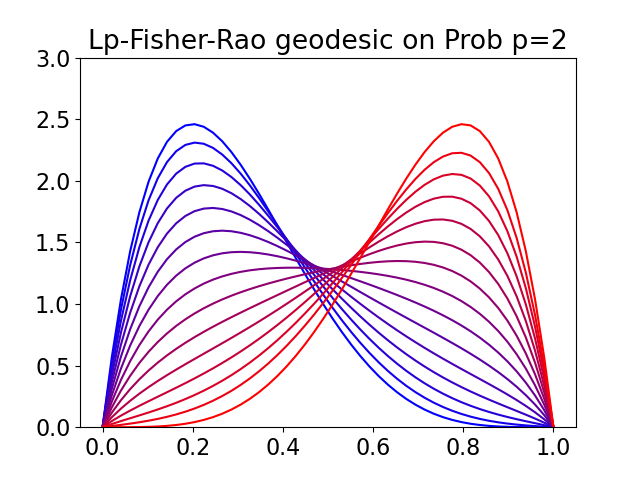}\\
    \includegraphics[width=0.28\linewidth]{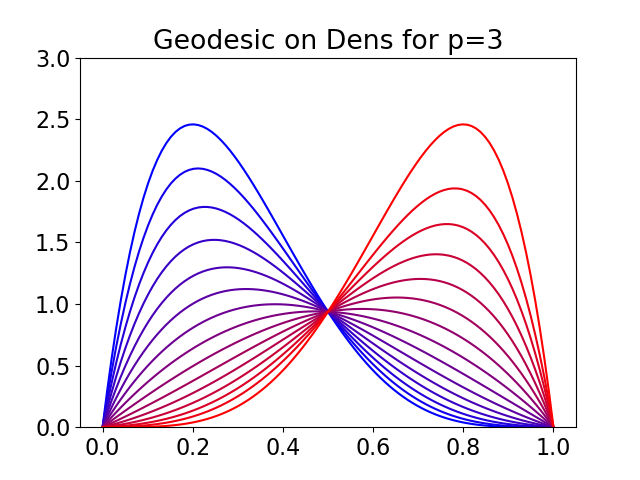}
    \includegraphics[width=0.28\linewidth]{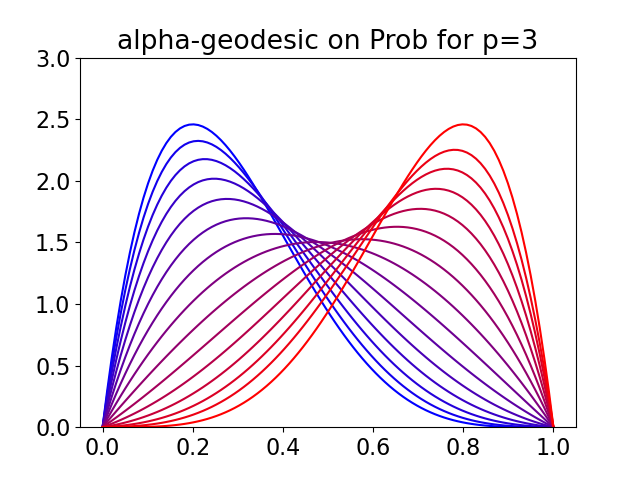}
    \includegraphics[width=0.28\linewidth]{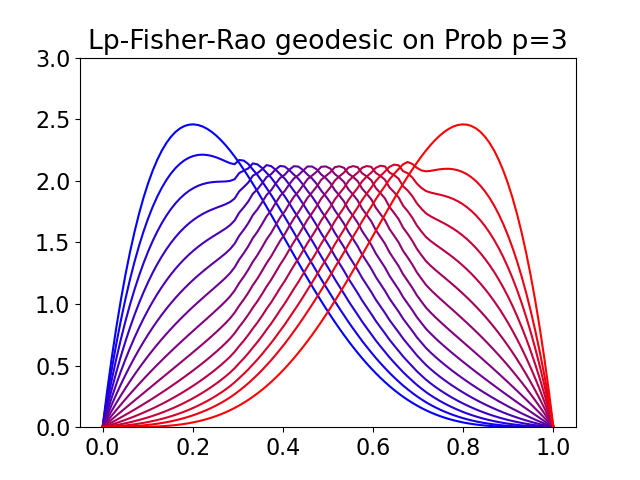}\\
    \includegraphics[width=0.28\linewidth]{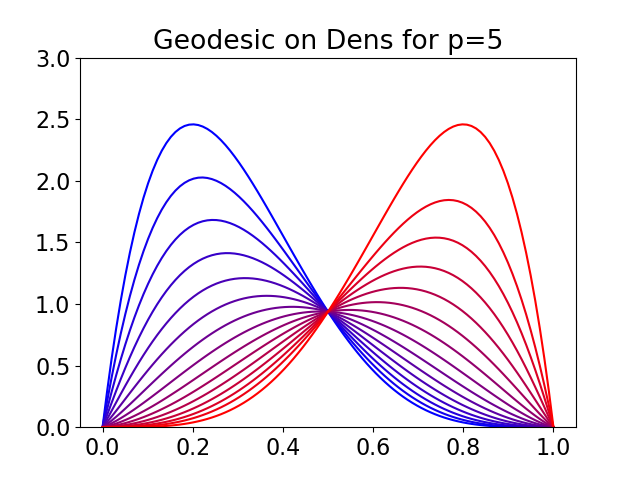}
    \includegraphics[width=0.28\linewidth]{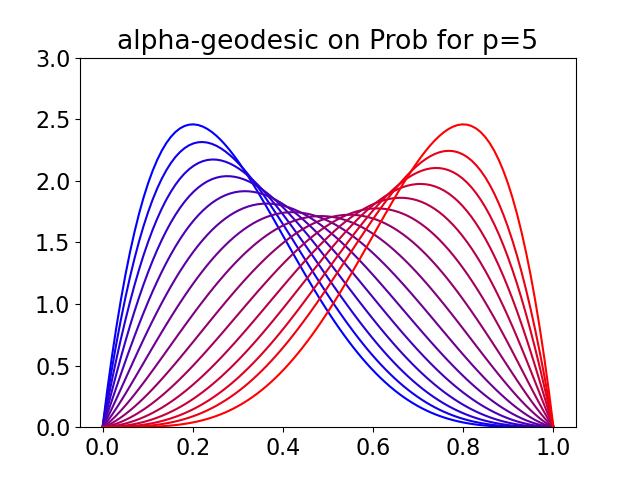}
    \includegraphics[width=0.28\linewidth]{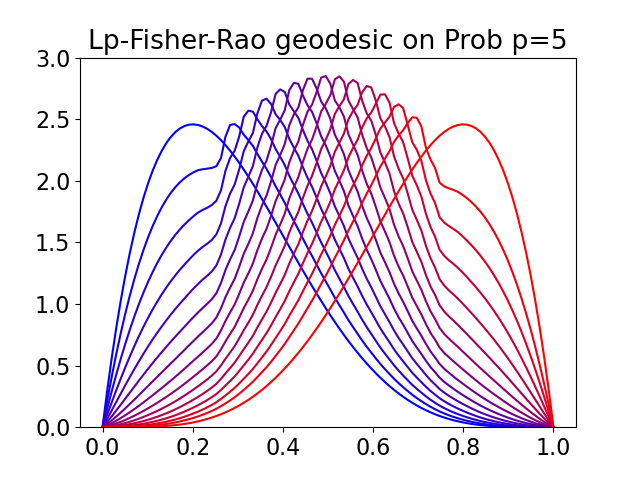}\\
    \includegraphics[width=0.28\linewidth]{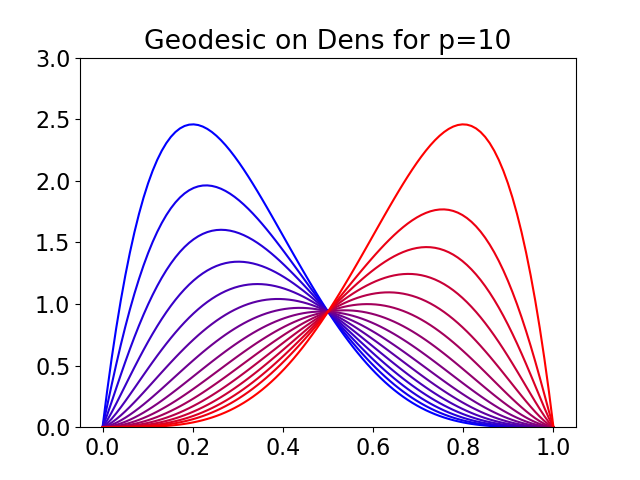}
    \includegraphics[width=0.28\linewidth]{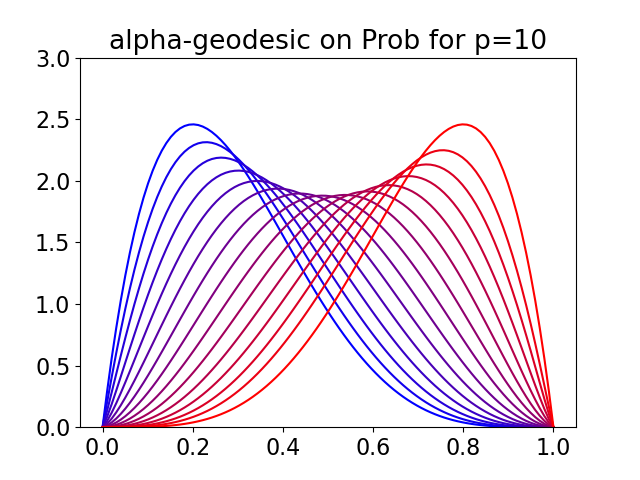}
    \includegraphics[width=0.28\linewidth]{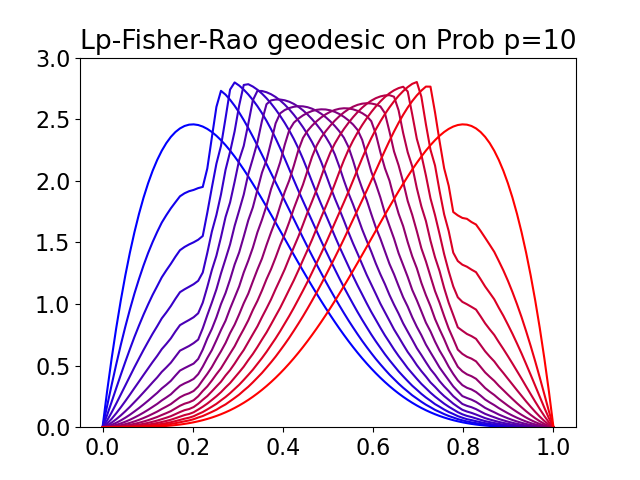}
    \caption{Different notions of geodesics between two probability distributions on $[0,1]$, for $p=2, 3, 5, 10$ from top to bottom, and corresponding values of $\alpha=1-2/p$. On the left: geodesics of $\Dens$ for the $L^p$-Fisher-Rao metric and the corresponding $\alpha$-connection. In the middle: $\alpha$-geodesics on $\Prob$. On the right: $L^p$-Fisher-Rao geodesics on $\Prob$. The last two notions coincide only for $p=2$.}
    \label{fig:geodesics}
\end{figure}

\section{Finite-dimensional geometry of parametric statistical models}\label{sec:finite-dim}

In this section we make the link with the finite-dimensional setting of parametric statistical models. Let us consider a finite-dimensional submanifold of $\ProbRn$ corresponding to a family of probability distributions on $\R^n$ that are absolutely continuous with respect to the Lebesgue measure, and whose densities are parametrized by a parameter $\theta$ belonging to an open subset $\Theta$ of $\R^d$:
\begin{align*}
    \mathcal P_\Theta = \{\mu(dx)=f(x,\theta)\,dx: \theta\in\Theta\}\subset \ProbRn.
\end{align*}
Here $x\in\R^n$ is the sample variable and $dx$ denotes the Lebesgue measure on $\R^n$. Then a tangent vector of $\mathcal P_\Theta$ at a given $\mu=f(\cdot,\theta) dx$ is of the form $a=\left.\frac{d}{dt}\right|_{t=0}\mu_t$, where $\mu_t=f(\cdot,\theta_t)dx$ with $t\mapsto\theta_t$ a curve in $\Theta$ such that $\theta_0=\theta$ and $\dot\theta_0=
u\in T_\theta\Theta$. 
Thus the tangent space at $\mu$ is
\begin{align}\label{basis-vectors}
T_\mu\mathcal P_\Theta&=\{a= \langle\grad_\theta f, u\rangle dx : u\in T_\theta\Theta \simeq \R^d\}\nonumber\\
&=\mathrm{span}\{e_1,\hdots,e_d\},
\end{align}
where $e_i=\frac{\partial f}{\partial\theta^i}dx$. 
Here $\grad_\theta$ denotes the gradient with respect to $\theta$ and $\langle\cdot,\cdot\rangle$ the Euclidean scalar product on $\R^d$. In all the sequel, we identify $\mathcal P_\Theta \simeq \Theta$ and $T_\mu\mathcal P_\Theta \simeq T_\theta \Theta \simeq \R^d$ via the one-to-one maps
\begin{equation}\label{bij}
\begin{aligned}
&\phi: \Theta \rightarrow \mathcal P_\Theta, \quad \theta \mapsto f(\cdot,\theta)dx,\\
&(\phi_*)_\theta: T_\theta\Theta\rightarrow T_{\phi(\theta)}\mathcal P_\Theta, \quad u \mapsto \langle \grad_\theta f,u\rangle dx.
\end{aligned}
\end{equation}

\subsection{The Fisher-Rao metric and the $\alpha$-connection}

The Fisher-Rao metric on the parameter space $\Theta$ is the Riemannian metric whose metric matrix is the Fisher information matrix
$$G(\theta)=\E\left[\grad_\theta\ell(X,\theta)\grad_\theta\ell(X,\theta)^\top\right].$$
Here $\E$ denotes the expectation taken with respect to the random variable $X$ of density $f(\cdot,\theta)$, and $\ell(x,\theta)=\log f(x,\theta)$ is the log-likelihood. 
\begin{definition}
    Given $\theta\in\Theta$ and $u,v\in T_\theta\Theta\simeq\R^d$, the Fisher-Rao metric is
    $$\GFR_\theta(u,v)=u^\top G(\theta)v=\E\left[\langle\grad_\theta\ell,u\rangle\langle\grad_\theta\ell,v\rangle\right],$$
    where $\langle\cdot,\cdot\rangle$ denotes the Euclidean scalar product on $\R^d$.
\end{definition}
The Fisher-Rao metric on the parameter space $\Theta$ is the pullback of the Fisher-Rao metric on the infinite-dimensional space $\ProbRn$ by the bijection $\phi$ defined by \eqref{bij}, i.e. for any $\theta\in\Theta$ and $u,v\in T_\theta\Theta$,
$$\GFR_{\phi(\theta)}(\phi_*u,\phi_*v)=\GFR_\theta(u,v),$$
and so both are denoted the same way.

Just like in the infinite-dimensional setting, the $\alpha$-connection on the parameter space can be defined using the $\alpha$-divergence.
\begin{definition}
The $\alpha$-connection on $\Theta$ is defined by its Christoffel symbols of the first kind (\cite{zhang2004divergence}, Eqn 2.9)
$$\left.\tilde\Gamma^{(\alpha)}_{ij,k}\right|_\theta=\GFR_\theta(\nabla^{(\alpha)}_{\partial_i}{\partial_j},\partial_k)=\left.-\frac{\partial^3}{\partial u^i \partial u^j \partial v^k}D^{(\alpha)}(\theta+u , \theta + v)\right|_{u=v=0}$$
where
$$D^{(\alpha)}(\theta, \theta')=\frac{4}{1-\alpha^2}\left(1 - \int f(x,\theta)^{\frac{1-\alpha}{2}}f(x,\theta')^{\frac{1+\alpha}{2}} dx\right)$$
is the $\alpha$-divergence. This yields the following formula in local coordinates, where $\partial_i$ denotes $\partial/\partial\theta^i$,
\begin{equation}\label{alpha-finite}
\tilde\Gamma_{ij,k}^{(\alpha)}=\E\left[(\partial_i\partial_j\ell + \frac{1-\alpha}{2} \partial_i\ell \partial_j\ell )\partial_k\ell\right].
\end{equation}
\end{definition}
The following result is well-known in the literature, and stated e.g. in \cite{amari2000methods} for spaces of probability distributions on a finite set.
\begin{theorem}[$\alpha$-connection on $\Theta$]\label{thm:finite-alpha}
For any $u,v\in T_\theta\Theta$, we have
$$\nablatilde^{(\alpha)}_uv = \mathrm{Proj}^{\mathrm{FR}}\left(\nablabar^{(\alpha)}_{\phi_*u}\phi_*v\right),$$
where $\nablatilde^{(\alpha)}$ and $\nablabar^{(\alpha)}$ denote the $\alpha$-connections on $\mathcal P_\Theta \simeq\Theta$ and $\ProbRn$ respectively, and $\mathrm{Proj}^{\mathrm{FR}}:T_{\phi(\theta)}\ProbRn\rightarrow T_\theta\Theta$ is the orthogonal projection with respect to the Fisher-Rao metric.
\end{theorem}
\begin{proof}
First notice that at any $\mu=\phi(\theta)\in\mathcal P_\Theta$, the orthogonal projection of a tangent vector $a\in T_\mu\ProbRn$ onto $T_\theta\Theta$ with respect to the Fisher-Rao metric $\GFR$ is given by 
\begin{equation}\label{proj-fr}
\mathrm{Proj}^{\mathrm{FR}}(a)=G^{ij}\GFR(a,e_j)\partial_i,
\end{equation}
where $(G^{ij})_{ij}$ is the inverse of the Fisher matrix. Indeed, the tangent space $T_\mu\mathcal P_\Theta$ is a $d$-dimensional vector space spanned by the tangent vectors $e_i=\partial_i f dx$ for $i=1,\hdots,d$, and so the orthogonal projection of $a\in T_\mu\ProbRn$ onto $T_\mu\mathcal P_\Theta$ is given by $u^ie_i$ where for $j=1,\hdots, d$, 
$$\GFR(a-u^ie_i, e_j)=0 \quad \text{i.e.} \quad \GFR(a,e_j)=u^i\GFR(e_i,e_j)=u^iG_{ij}.$$
The $\alpha$-connection on $\ProbRn$ is given by 
 $$(\nablabar^{(\alpha)}_ab)_\mu=D_\mu b(a)-\frac{1+\alpha}{2}\left(\frac{a}{\mu}\frac{b}{\mu}\mu - \left(\int \frac{a}{\mu}\frac{b}{\mu}\mu\right)\mu\right),$$ 
where $D_\mu b(a)$ is the directional derivative of the vector field $b$ in the direction of the vector $a_\mu$. Let $\partial_i$ denote partial derivative with respect to $\theta^i$ for all $i=1,\hdots,d$. For vector fields on the finite-dimensional manifold $\mathcal P_\Theta$,
$$a=\phi_*u=u^ie_i, \quad b=\phi_*v=v^je_j,$$
and at $\mu=\phi(\theta)$, we get since $e_i=\partial_if dx$,
\[
\begin{split}
D_\mu b(a)&=D_\theta v(u)=\partial_i(\partial_jf v^j)u^i dx = (\partial_i\partial_j fu^iv^j + \partial_jfu^i\partial_iv^j) dx\\
&=((\partial_i\partial_j \ell + \partial_i\ell\partial_j\ell)u^iv^j + \partial_j\ell u^i\partial_iv^j) f dx,
\end{split}
\]
where in the last equality we used the equality $\partial_i\partial_j\ell = \partial_i\partial_jf/f - \partial_i\ell \partial_j \ell$.
Since 
$$\frac{a}{\mu}\frac{b}{\mu}\mu = \partial_i\ell\partial_j\ell u^iv^j f dx,$$
we obtain $\nablabar^{(\alpha)}_ab=h dx$ where
\begin{align*}
h/f=\partial_j\ell u^i\partial_iv^j + \left(\partial_i\partial_j\ell + \frac{1-\alpha}{2}\partial_i\ell\partial_j\ell\right)u^iv^j - \frac{1+\alpha}{2}\GFR(a,b).
\end{align*}
Remembering that $\GFR(hdx,kdx)=\E(hk/f^2)$ and since $\E(\partial_m\ell)=0$, we obtain using \eqref{alpha-finite},
\begin{align*}
\GFR(\nablabar^{(\alpha)}_ab,e_m)&=\E[\partial_j\ell\partial_m\ell]u^i\partial_iv^j + \E\left[(\partial_i\partial_j\ell + \frac{1-\alpha}{2}\partial_i\ell\partial_j\ell)\partial_m\ell\right]u^iv^j\\
&=G_{jm} u^i\partial_iv^j+ \tilde\Gamma^{(\alpha)}_{ij,m}u^iv^j.
\end{align*}
Finally, using \eqref{proj-fr}, we see that 
$$\mathrm{Proj}^{\mathrm{FR}}(\nablabar^{(\alpha)}_ab) = G^{km}\GFR(\nablabar^{(\alpha)}_ab,e_m)\partial_k= \left(u^i\partial_iv^k+ G^{km}\tilde\Gamma^{(\alpha)}_{ij,m}u^iv^j\right)\partial_k=\nablatilde^{(\alpha)}_{u}v,$$
which concludes the proof.
\end{proof}

\subsection{The $L^p$-Fisher-Rao metric}

We now introduce a finite-dimensional version of the Finsler $L^p$-Fisher-Rao metric.
\begin{definition}
Given $\theta\in\Theta$ and $v\in T_\theta\Theta$ we define the $L^p$-Fisher-Rao metric on $\Theta$ as
\begin{align}\label{finite-lp-fr}
    F_p(\theta,v)=\left(\E|\langle \grad_\theta \ell(X,\theta), v\rangle|^p\right)^{1/p}.
\end{align}
Here $\langle\cdot,\cdot\rangle$ denotes the Euclidean scalar product on $\R^d$, $\E$ denotes the expectation taken with respect to the random variable $X$ of density $f(\cdot,\theta)$, and $\ell(x,\theta)=\log f(x,\theta)$ is the log-likelihood.
\end{definition}
The metric \eqref{finite-lp-fr} on the parameter space $\Theta$ coincides with the Finsler metric induced on $\mathcal P_\Theta$ by the $L^p$-Fisher-Rao metric \eqref{lp-fisher-rao} through the identification $\mathcal P_\Theta \simeq \Theta$, which is why they are denoted the same way. Indeed, for any $(\theta,v)\in T\Theta$, 
    \begin{align*}
    F_p(\phi(\theta),\phi_*v)=\int\left|\frac{\langle \grad_\theta f(x,\theta),v\rangle}{f(x,\theta)}\right|^p f(x,\theta)dx
    =\E|\langle \grad_\theta \ell(X,\theta),v\rangle|^p=F_p(\theta,v).
    \end{align*}
\begin{lemma}[Induced Chern connection on $\Theta$]
The Chern connection associated to the $L^p$-Fisher-Rao metric on $\Theta$ is given by
\begin{multline}\label{finite-chern-1}
(\nablatilde^v_uv)^m= u(v^m) + (g^v)^{mk}\big( g^{\phi_*v}(\omega(u,v),e_k) + C^{\phi_*v}(\omega(v,v),\phi_*u,e_k)\\
-(g^v)^{ij}g^{\phi_*v}(\omega(v,v),e_i)C^{\phi_*v}(\phi_*u,e_j,e_k)\big),
\end{multline}
where $g$ and $C$ respectively denote the Riemannian metric \eqref{riem-metric} and Cartan tensor \eqref{cartan-tensor} induced by the $L^p$-Fisher-Rao metric, $(g^v)_{ij}=g^{\phi_*v}(e_i,e_j)$ and $(e_i)_i$ are the basis vectors \eqref{basis-vectors} of $T_\mu\mathcal P_\Theta$ and
\begin{equation}\label{omega}
\omega(u,v)=\omega_{ij}u^iv^i \quad \text{with} \quad \omega_{ij}=\left(\partial_i\partial_j\ell+\frac{1}{p}\partial_i\ell\partial_j\ell\right)fdx.
\end{equation}
\end{lemma}
\begin{proof}
Let $a=\phi_*u$, $\nu=\phi_*v$, $\alpha=1-2/p$, and $\nabla^{(\alpha)}$ be the $\alpha$-connection on $\DensRn$. Similarly to the orthogonal projection with respect to the Fisher-Rao metric \eqref{proj-fr}, the orthogonal projection on $T\Theta$ with respect to $g^\nu$ is given by
\begin{equation}\label{proj-pfr}
\mathrm{Proj}^\nu(a)=(g^v)^{ij} g^\nu(a, e_j)\partial_i.
\end{equation}
Let us denote
$$(\nabla^{(\alpha)}_a\nu)^\top:=\phi_*\mathrm{Proj}^\nu\left(\nabla^{(\alpha)}_a\nu\right)=(g^v)^{ij} g^\nu(a, e_j)e_i, \quad (\nabla^{(\alpha)}_a\nu)^\perp:=(\nabla^{(\alpha)}_a\nu) - (\nabla^{(\alpha)}_a\nu)^\top.$$
We define the connection $\nablatilde$ by
\begin{equation}\label{nabla}
\phi_*(\nablatilde_uv):=(\nabla^{(\alpha)}_a\nu)^\top + (g^v)^{mk} C^\nu\left((\nabla^{(\alpha)}_\nu\nu)^\perp, a, e_k\right)e_m.
\end{equation}
Let us show that $\nablatilde$ is the Chern connection $\nablatilde^v$ on $\Theta$, by showing once again that it verifies the generalized Koszul formula \eqref{koszul}. Using the notations $g^v(u,w)=g^{\phi_*v}(\phi_*u,\phi_*w)$, $C^v(u,w,z)=C^{\phi_*v}(\phi_*u,\phi_*w,\phi_*z)$ and the fact that $C^v(v,\cdot,\cdot)=0$, the generalized Koszul formula can be written
\begin{align*}
2g^v(\nablatilde_uv,w)&=ug^v(v,w)+vg^v(w,u)-wg^v(u,v)\\
&+g^v([u,v],w)-g^v([v,w],u)+g^v([w,u],v)-2C^v(\nabla_vv,u,w).
\end{align*}
Recalling that $\nabla^{(\alpha)}$ is the Chern connection on $\Dens$ and noticing that $\phi_*(\nabla_vv)=(\nabla^{(\alpha)}_\nu\nu)^\top$, the previous equation is satisfied if and only if
$$g^\nu\left((g^v)^{ij} C^\nu\left((\nabla^{(\alpha)}_\nu\nu)^\perp, a, e_j\right)e_i,\phi_*w\right)=C^\nu\left((\nabla^{(\alpha)}_\nu\nu)^\perp,a,\phi_*w\right),$$
which is easily checked to be true using the fact that $ g^\nu(e_i,e_j)=(g^v)_{ij}$. To obtain the desired formula for $\nablatilde^v$, we write the $\alpha$-connection in coordinates, through the same computations as in the proof of Theorem \ref{thm:finite-alpha}
\begin{align*}
\nabla^{(\alpha)}_a\nu=D\nu.a-\frac{1+\alpha}{2}\frac{a}{\mu}\frac{\nu}{\mu}\mu=u^i\partial_iv^je_j + \left(\partial_i\partial_j\ell + \frac{1-\alpha}{2}\partial_i\ell\partial_j\ell\right)u^iv^jfdx=u(v^j)e_j+\omega(u,v).
\end{align*}
Using \eqref{proj-pfr} we obtain
\begin{align*}
(\nabla^{(\alpha)}_a\nu)^\top&=u(v^m)e_m+(g^v)^{mk} g^\nu(\omega(u,v),e_k)e_m\\
(\nabla^{(\alpha)}_\nu\nu)^\perp&=w(v,v)-(g^v)^{ij} g^\nu(\omega(v,v),e_i)e_j
\end{align*}
which injected into \eqref{nabla} gives the desired result.
\end{proof}
\begin{remark}
Like in infinite dimensions (see Remark \ref{rem:chern-prob}), when all entries are the same, the Chern connection on $\mathcal P_\Theta\simeq\Theta$ is the orthogonal projection of the $\alpha$-connection $\nabla^{(\alpha)}$ on $\DensRn$, for $\alpha=1-\frac{2}{p}$, with respect to $g^{\phi_*v},$ the Riemannian metric \eqref{riem-metric} induced by the $L^p$-Fisher-Rao metric
\begin{equation}\label{finite-chern-2}
\nablatilde^v_vv=\mathrm{Proj}^{\phi_*v}\left(\nabla^{(\alpha)}_{\psi_*v}\phi_*v\right)=\left(v(v^m) + (g^v)^{mk} g^{\phi_*v}(\omega(v,v),e_k)\right)\partial_m.
\end{equation}
\end{remark}
\begin{theorem}[Geodesic equation on $\Theta$]
The geodesic equation of the $L^p$-Fisher-Rao metric on the space $\mathcal P_\Theta$ is given by
\begin{align}\label{eq:geodesicTheta} 
\ddot\theta^m + (g^{\dot\theta})^{mk} g^{\phi_*\dot\theta}(\omega(\dot\theta,\dot\theta),e_k)=0,
\end{align}
where $(e_i)_i$ are the basis vectors \eqref{basis-vectors} and
$\omega$ is defined by \eqref{omega}.
\end{theorem}
\begin{proof}
This results directly from Lemma~\ref{lem:finsler-chern} in Appendix~\ref{sec:finsler} and writing $\nablatilde^{\dot\theta}_{\dot\theta}\dot\theta=0$ in local coordinates using \eqref{finite-chern-2}.
\end{proof}

\begin{example}[Normal distributions]

\begin{figure}\label{fig:normal}
\includegraphics[width=0.4\linewidth]{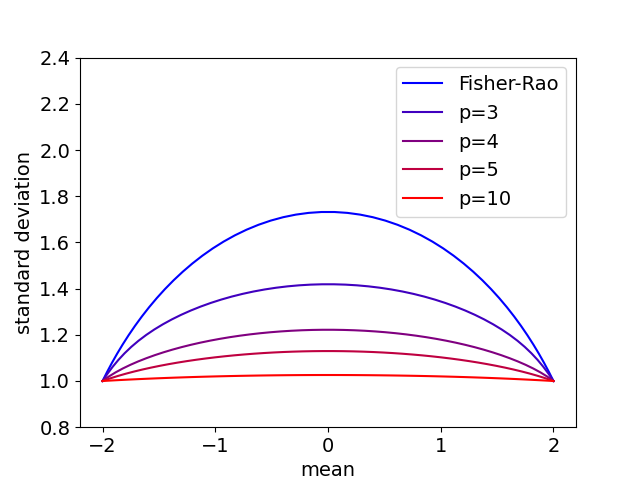}
\includegraphics[width=0.4\linewidth]{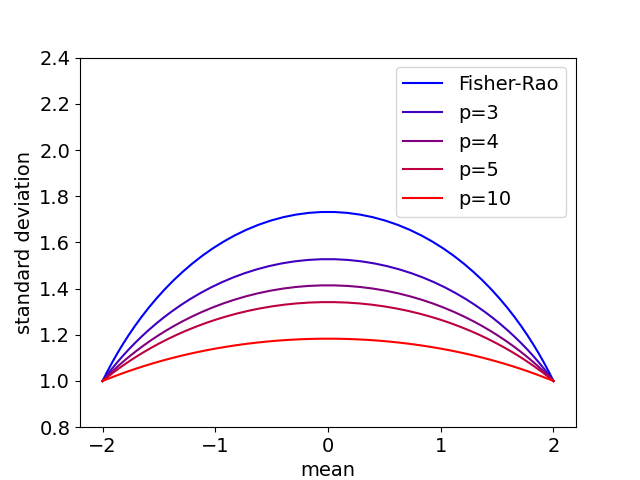}
\caption{Comparison of the geodesics between two normal distributions shown in the parameter space for the $L^p$-Fisher-Rao metric (left) and for the $\alpha$-connection (right), for different values of $p$ and the corresponding values of $\alpha=1-2/p$. The geodesics all start at the normal distribution of parameters $m_0=-2$, $\sigma_0=1$, and end at the normal distribution of parameters $m_1=2$, $\sigma_1=\sigma_0=1$.}
\end{figure}

Let us consider the example of univariate normal distributions, parametrized by mean and standard deviation. The parameter space is the upper half-plane $\Theta=\R\times\R_+^*$, and the Fisher-Rao metric, after a change of coordinates, is the hyperbolic metric of the Poincaré half-plane. The family of Riemannian metrics induced by the Finsler $L^p$-Fisher-Rao metric are also Poincaré metrics: for any $\theta=(m,\sigma)\in \Theta$ and $v\in \R^2$, the metric matrix is given by
$$g^v_{(m,\sigma)}=\frac{1}{\sigma^2} g_0^v$$
where $g_0^v$ does not depend on $m$ and $\sigma$. In order to compute geodesics for the $L^p$-Fisher-Rao metric, one can solve the geodesic equation \eqref{eq:geodesicTheta}, using the following densities with respect to a given $\mu(dx)=\frac{1}{\sqrt{2\pi}\sigma}\exp(-\frac{(x-m)^2}{2\sigma^2})dx$: the basis vectors of the tangent plane $T_\mu\mathcal P_\Theta$ are given by
$$\frac{e_1}{\mu}=\frac{1}{\sigma}z, \quad \frac{e_2}{\mu}=\frac{1}{\sigma}(z^2-1) \quad \text{with}\quad z:=\frac{x-m}{\sigma}$$
and for a given curve $t\mapsto \theta(t)=(m(t),\sigma(t))$,
\begin{align*}
\frac{\dot\theta}{\mu}&=\frac{1}{\sigma}(\dot m z + \dot\sigma(z^2-1))\\
\frac{\omega(\dot\theta,\dot\theta)}{\mu}&=\frac{1}{\sigma^2}\left((-1+\frac{1}{p}z^2)\dot m^2+(1-3z^2+\frac{1}{p}(z^2-1)^2)\dot\sigma^2+2(-2z+\frac{1}{p}z(z^2-1))\dot m \dot\sigma\right).
\end{align*}
The $L^p$-Fisher-Rao geodesic can be compared to the solutions of the geodesic equation of the $\alpha$-connection for $\alpha=1-2/p$:
$$\ddot m - 2\frac{1+\alpha}{\sigma} \dot m \dot\sigma=0, \qquad \ddot\sigma + \frac{1-\alpha}{2\sigma}\dot m^2 - \frac{1+2\alpha}{\sigma}\dot\sigma^2=0.$$
In both cases, we solve the geodesic ODE with boundary constraints in Python for a discretization of $50$ time steps, using the dedicated function in Scipy\footnote{\url{https://docs.scipy.org/doc/scipy/reference/generated/scipy.integrate.solve_bvp.html}}, which implements a fourth order collocation algorithm. We plot in Figure~\ref{fig:normal} the $L^p$-Fisher-Rao geodesics for several values of $p$ as well as the $\alpha$-geodesics for the corresponding values of $\alpha$. As expected, these geodesics do not coincide, except for $p=2$, where we retrieve the Fisher-Rao metric.
\end{example}

\appendix
\section{Infinite dimensional Finsler geometry}\label{sec:finsler}

In this appendix we will present several key definitions of Finsler geometry in the infinite dimensional setting. We will base our definitions on their counterparts from classical finite dimensional Finsler geometry, see eg.~\cite{bao2000introduction, chern2005riemann, rademacher2004nonreversible}. 

In the following let $\mathcal M$ be an infinite dimensional, Fr\'echet manifold with tangent bundle $T\mathcal M$.

\begin{definition}[Finsler structure]\label{def:Finsler}
A Finsler structure on $\mathcal M$ is a function $F: T\mathcal M \to [0,\infty)$, that is smooth on the complement of the zero section of $T\mathcal M$ and satisfies for all $x\in\mathcal M$ and $X,Y\in T_x\mathcal M$
\begin{enumerate}[label=(\alph*)]
\item $F(\lambda Y)=\lambda F(Y)$ for all $\lambda>0$; 
\item $F(Y)\geq 0$ and $F(Y)=0$ if and only if $Y=0$.
\item \label{eq:subadd} 
$F(X+Y)\leq F(X)+F(Y)$.
\end{enumerate}
The Finsler norm $F$ is called strongly convex if we have in addition
\begin{enumerate}[label=(\alph*),resume]
\item \label{eq:strongconvex}For any $0\neq V\in T_x\mathcal M$
the Hessian matrix $g^{V}$ of $F^2$ at $V$ exists and is positive definite, where
\begin{align*}
g^{V}(X,Y):=\frac{1}{2}\frac{\partial^2}{\partial s\partial t}\left[F^2(V+sX+tY)\right]_{s=t=0}.
\end{align*}
\end{enumerate}
\end{definition}
\begin{remark}
It can be shown that the strong convexity condition~\ref{eq:strongconvex} implies the subadditivity condition~\ref{eq:subadd} and several modern textbooks require strong convexity instead of subadditivity in the definition of a Finsler metric as this allows to develop several concepts of Riemannian geometry in the Finslerian setting. We choose to not assume this stronger condition as our main example, the $L^p$-Fisher-Rao metric, is not strongly convex.  
\end{remark}
\begin{remark}[Weak and strong Finsler metrics]
For each $x\in \mathcal M$ the Finsler metric $F$ induces a topology on $T_x \mathcal M$ and in finite dimensions this topology coincides with the original manifold topology. In infinite dimensions this is not the case and we will distinguish between two different types of Finsler metrics:  strong Finsler metric, for which $F_{x}$ induces the locally convex topology on $T_x M$ and weak Finsler metrics, where the induced topology can be weaker than the locally convex topology.  If $\mathcal M$ is not a Banach manifold then any Finsler metric on $\mathcal M$ can only be a weak Finsler metric.
\end{remark}

Similarly as a Riemannian metric a Finsler structure $F$ on a manifold $\mathcal M$ defines a length structure on the set of piece wise smooth curves and thus one can define a corresponding path length distance:
\begin{definition}
Let $c:[a,b]\to\mathcal M$ be a piece wise smooth curve.
The length of $c$ with respect to $F$ is defined as
\begin{align*}
L_{F}(c):=\int_a^b F(c(t),\dot{c}(t)) dt.
\end{align*}
For any pair of points $x,y\in \mathcal M$ we consider the induced geodesic distance function  
\begin{align*}
d_{F}(p,q):=\operatorname{inf}_{c} L_{F}(c),
\end{align*}
where the infimum is calculated over the set of a piece wise smooth curves that connect $x$ to $y$. Similar as in Riemannian geometry one can show that minimizing the length is equivalent to minimizing the energy, which is defined as
\begin{align}\label{eq:finslerenergy}
E_{F}(c):=\int_a^b F^2(c(t),\dot{c}(t)) dt.
\end{align}
\end{definition}

\begin{remark}[Vanishing Geodesic distance]
It is easy to see that the geodesic distance functions is symmetric and satisfies the triangle inequality. In general, for weak Finsler metrics, it does not satisfy the non-degenracy property -- $d_{F}(x,y)=0$ if and only if $x=y$ for  Finsler metrics. Indeed, even in the Riemannian case, several examples have been encountered where the geodesic distance can be degenerate or even vanishes identically, see eg.~\cite{eliashberg1993bi,michor2005vanishing,bauer2020vanishing,jerrard2019vanishing}. 
\end{remark}

Next we will introduce two important concepts from Finsler geometry: 
 the Cartan tensor, which was introduced by E.~Cartan~\cite{cartan1934espaces} to evaluate the differences between Finsler metrics and Riemannian metrics,  and the Chern connection, which is a generalization of the Levi-Civita connection on a Finsler manifold.

Note, that the definition of these two objects requires that the Finsler metric is strongly convex. As the $L^p$-Fisher-Rao metric, studied in the following sections, will not satisfy this property several of the calculations in these parts have to be taken with caution and should be thus  understood only formally. 
\begin{definition}[Cartan Tensor and Chern connection]
Let $(\mathcal M, F)$ be a Finsler manifold, where $F$ is assumed to satisfy the strong convexity assumption. For any nonzero tangent vector $V\in T_x\mathcal M$, the Cartan tensor is defined as the symmetric trilinear form
\begin{align*}
C^V(X,Y,Z):=\frac{1}{4}\frac{\partial^3}{\partial s\partial t\partial r}\left[F^2(V+sX+tY+rZ)\right]_{s=t=r=0},
\end{align*}
and the Chern connection, if it exists, is the unique affine, torsion-free connection $\nabla^V$ that
is almost metric, that is for vector fields $X,Y,Z$ we have
\begin{align*}\label{almost_metric}
X g^V(Y,Z)= g^V(\nabla^V_X Y,Z)+g^V(Y,\nabla^V_X Z)+2C^V(\nabla^V_X V,Y,Z).
\end{align*}
\end{definition}
\begin{remark}
In the above definition of the Chern-connection we have added the assumption on it's existence. This is additional assumption is not necessary in finite dimensions, but is entirely an infinite dimensional phenomenon, see eg.~\cite{bauer2014homogeneous} where the authors studied a Riemannian metric on a group of diffeomorphisms such that the corresponding Levi-Civita connection does not exist. 
\end{remark}

The next Lemma, which will be of importance when we show the equivalence of the Chern-connection of the $L^p$-Fisher-Rao metric and the $\alpha$-connection on $\operatorname{Dens}(M)$, provides a generalized Koszul-formula for the Chern-connection:
\begin{lemma}\label{lem:koszul}
Let $(\mathcal M, F)$ be a Finsler manifold, where $F$ is assumed to satisfy the strong convexity assumption. For every non-zero vector field $V\in T_x\mathcal M$, the Chern connection, if it exists, satisfies the  generalized Koszul formula
\begin{equation}\label{koszul}
\begin{aligned}
&2g^V(\nabla^V_X Y,Z)=X g^V(Y,Z)+ Y g^V(Z,X)-Z g^V(X,Y)\\
&\qquad + g^V([X,Y],Z)-g^V([Y,Z],X)+g^V([Z,X],Y)\\
&\qquad -2C^V(\nabla^V_X V,Y,Z)-2C^V(\nabla^V_Y V,Z,X)+2C^V(\nabla^V_Z V,X,Y) 
\end{aligned}
\end{equation}
\end{lemma}
\begin{proof}
    The proof of this result follows exactly as in the finite dimensional situation, see eg.~\cite{chern2005riemann}.
\end{proof}
The next results shows that the Chern-connection, similarly to the Levi-Civita connection in Riemannian geometry, describes the locally minimizing curves. 
\begin{lemma}\label{lem:finsler-chern}
Let $(\mathcal M, F)$ be a Finsler manifold, where $F$ is assumed to satisfy the strong convexity assumption. Assume in addition that the Chern connection $\nabla$ exists. 
Then the critical points of the energy functional~\eqref{eq:finslerenergy} are describe by the geodesic equation
\begin{equation}
\nabla^{{c_t}}_{c_t}c_t=0.
\end{equation}
\end{lemma}
\begin{proof}
Assuming the existence of the Chern connection, this follows exactly as  in finite dimensions, see eg.~\cite{chern2005riemann,rademacher2004nonreversible}.    
\end{proof}

\bibliography{bibliography}
\bibliographystyle{abbrv}

\end{document}